\documentclass[12pt]{amsart}
\usepackage{fullpage,xcolor,graphicx}
\usepackage[OT2,T1]{fontenc}
\usepackage{hyperref}
\usepackage{pifont}
\usepackage{amsthm}
\usepackage{amsfonts}
\usepackage{amssymb}
\usepackage[mathscr]{euscript}
\usepackage[all]{xy}
\usepackage{amsmath}
\usepackage{epsfig}
\usepackage{latexsym}
\usepackage{stackengine}
\usepackage{caption}
\usepackage{subcaption}

\usepackage[OT2,OT1]{fontenc} \newcommand\cyr
{
\renewcommand\rmdefault{wncyr} \renewcommand\sfdefault{wncyss} \renewcommand\encodingdefault{OT2} \normalfont
\selectfont
}
\DeclareTextFontCommand{\textcyr}{\cyr} 

\newcommand*\wbar[1]{
  \hbox{ \kern-0.2em%
    \vbox{%
      \hrule height 0.5pt  
      \kern0.25ex
      \hbox{%
        \kern-0.15em
        \ensuremath{#1}%
        \kern-0.05em
      }%
    }%
  \kern0.05em}%
} 

\newcommand*\wbarnew[1]{
  \hbox{ \kern-0.2em%
    \vbox{%
      \hrule height 0.5pt  
      \kern0.25ex
      \hbox{%
        \kern-0.35em
        \ensuremath{#1}%
        \kern-0.05em
      }%
    }%
  \kern0.05em}%
}

\newcommand{\Zh}{\hbox{\hspace{-5.8mm} \textcyr{Zh}}} 
\newcommand{\zh}{\hbox{\hspace{-5.8mm} \textcyr{zh}}} 

\newtheorem{theorem}{Theorem}[subsection]
\newtheorem{lemma}[theorem]{Lemma}
\newtheorem{proposition}[theorem]{Proposition}
\newtheorem{corollary}[theorem]{Corollary}
\newtheorem{thm}{Theorem}

\theoremstyle{definition}
\newtheorem{definition}[theorem]{Definition}
\newtheorem{example}[theorem]{Example}
\newtheorem*{acknowledgments*}{Acknowledgments}{}
\theoremstyle{remark}

\newcommand{\vlk}{\operatorname{{\it v}\ell{\it k}}}

\makeatletter
\@namedef{subjclassname@2020}{\textup{2020} Mathematics Subject Classification}
\makeatother

\author[M. Chrisman]{Micah Chrisman}
\address{Department of Mathematics, The Ohio State University, Columbus, Ohio, USA}
\email{chrisman.76@osu.edu}
\urladdr{https://micah46.wixsite.com/micahknots}

\author[R. Todd]{Robert G. Todd}
\address{Department of Natural and Applied Sciences, Mount Mercy University, Cedar
Rapids, Iowa, USA}
\email{rtodd@mtmercy.edu}

\begin{document}

\title{Characterization and Further Applications of \\ the Bar-Natan $\!\!\!\!\!\!\Zh$-Construction}

\begin{abstract}
 Bar-Natan's $\Zh$-construction associates to each $n$ component virtual link diagram $L$ an $(n+1)$ component virtual link diagram $\Zh(L)$. If $L_0,L_1$ are equivalent virtual link diagrams, then $\Zh(L_0),\Zh(L_1)$ are equivalent as semi-welded links. The importance of the $\Zh$-construction is that it unifies several classical knot invariants with virtual knot invariants. For example, the generalized Alexander polynomial of a virtual link diagram $L$ is identical to the usual multi-variable Alexander polynomial of $\Zh(L)$. From this it follows that the generalized Alexander polynomial is a slice obstruction: it vanishes on any knot concordant to an almost classical knot. Our main result is a characterization theorem for the $\Zh$-construction in terms of almost classical links. Several consequences of this characterization are explored. First, we give a purely geometric description of the $\Zh$-construction. Secondly, the $\Zh$-construction is used to obtain a simple derivation of the Dye-Kauffman-Miyazawa polynomial. Lastly, we show that every quandle coloring invariant and quandle 2-cocycle coloring invariant can be extended to a new invariant using the $\Zh$-construction.
\end{abstract}

\maketitle

\section{Introduction}

\subsection{Motivation} \label{sec_motivate} Many invariants of virtual knots appear to be more closely related to invariants of multi-component classical links than to those of classical knots. For example, both classical multi-string links and virtual knots have interesting families of integer valued finite-type concordance invariants. For string-links, we have Milnor's $\bar{\mu}$-invariants \cite{HM}. For virtual knots, we have the odd writhe, Henrich-Turaev polynomial, affine index polynomial, and many others \cite{bcg1,c_ext}. In contrast, classical knots have only the Arf invariant \cite{ng}. As a second illustration, consider the generalized Alexander polynomial $\wbar{H}^{(0)}_K(x,y)$ of a virtual knot $K$ \cite{silwill,saw}. This is the first polynomial in a sequence $\wbar{H}^{(k)}_K(x,y)$ of polynomials obtained from the elementary ideal theory of a certain extension of the fundamental group. This sequence is analogous to the multi-variable Alexander polynomial $\Delta^{(k+1)}_L(x,y)$ of a classical link $L$. Indeed, $\wbar{H}^{(0)}_K(x,y)$ vanishes on an almost classical knot $K$, while $\Delta^{(1)}_L(x,y)$ is trivial when $L$ is a classical boundary link. If $K$ is an almost classical knot, $\wbar{H}_K^{(1)}(x,y)$ is essentially the usual Alexander polynomial $\Delta_K(t)$ \cite{silwill}. Futhermore, $K$ gives a two component \emph{classical} boundary link $L=K \sqcup J$ in $\mathbb{S}^3$ such that $\Delta_K(t)$ is a specialization of $\Delta_L^{(2)}(x,y)$ \cite{chrisman_todd}.  
\newline
\newline
In a series of talks (see e.g. \cite{dror}), Bar-Natan introduced a simple and elegant method that realizes this loose analogy between virtual knots and classical links. Bar-Natan's $\Zh$-\emph{construction} associates to each $n$ component virtual link diagram $D$ an $(n+1)$ component virtual link diagram $\Zh(D)=D \cup \omega$. The map is functorial in the sense that if $D_0$ and $D_1$ are equivalent virtual link diagrams, then $\Zh(D_0)$ and $\Zh(D_1)$ are semi-welded equivalent. The generalized Alexander polynomial of $D$ is then exactly the usual multi-variable Alexander polynomial of $\Zh(D)$. For a virtual knot diagram $D$, the extra variable $y$ in the  generalized Alexander polynomial $\wbar{H}_K^{(0)}(x,y)$ corresponds to a meridian of the added component $\omega$. By a theorem of Mellor \cite{mellor}, the odd writhe, Henrich-Turaev polynomial, and affine index polynomial can all be obtained from various specializations of the generalized Alexander polynomial. Moreover, the $\bar{\mu}$-invariants of $\Zh(D)$ define a set of extended Milnor invariants for $D$, called the $\wbar{\zh}$\emph{-invariants} \cite{c_ext}. In parallel with the case of classical links, the vanishing of the $\wbar{\zh}$-invariants implies the vanishing of the generalized Alexander polynomial. The $\Zh$-construction thus allows for the unification of a whole host of virtual knot invariants into a single hierarchy. Furthermore, this hierarchy uses only the standard technology of classical link theory. 
\newline
\newline
The reorganization provided by the $\Zh$-construction allows for classical geometric tools to be directly applied to the study of virtual knots. We mention a few examples in order to motivate our further study of the $\Zh$-construction. Consider the proof that the multi-variable Alexander polynomial vanishes on any link in $\mathbb{S}^3$ that is concordant to a boundary link (see e.g. \cite{hill}). Because the $\Zh$-construction is functorial under concordance \cite{bbc}, this argument can be transferred over to the virtual setting. In particular, theorems about 2-component links now become theorems about 1-component virtual links, i.e. virtual knots. The classical theorem now translates to the following: if a virtual knot is concordant to an almost classical knot, then its generalized Alexander polynomial is trivial. Similarly, the vanishing of the $\wbar{\mu}$-invariants on links concordant to boundary links translates to the vanishing of the $\wbar{\zh}$-invariants on any virtual knot concordant to an almost classical knot \cite{c_ext}. These new slice obstructions obtained from the $\Zh$-construction are quite powerful; there are many examples of non-slice virtual knots having trivial Rasmussen invariant but non-trivial$\wbar{\zh}$-invariants. 
\newline
\newline
The $\Zh$-construction thus serves dual purposes. By unifying virtual knot invariants with classical link theory, the $\Zh$-construction opens a bridge over which classical geometric tools can be ported to virtual knot theory. In the process of this translation, almost classical links appear to play a central role. The primary aim of this paper is to clarify the curious relationship between almost classical links and the $\Zh$-construction. Consequently, we find further instances of realization and extension of virtual link invariants using the $\Zh$-construction.

\subsection{Main results: characterization of $\Zh$}  Recall that almost classical links were originally defined by Silver-Williams \cite{silwill} as those virtual links having an Alexander numerable diagram. Our main result is a characterization theorem for the $\Zh$-construction in terms of Alexander numberings.  For any $n$ component virtual link diagram $D$, we consider a family of $(n+1)$ component virtual link diagrams $D \cup \gamma$ which satisfy both of the following properties (see Definition \ref{defn_zh_system} ahead for more details).
\begin{enumerate}
\item Any classical crossings involving the $\gamma$ component are over-crossings of $\gamma$ with $D$.
\item The arcs in the subdiagram $D$ of $D \cup \gamma$ are Alexander numerable.
\end{enumerate}
This will be called an \emph{Alexander system for} $D$. We will show that for any $D$, an Alexander system for $D$ can be obtained from $\Zh(D)=D \cup \omega$ by simply changing the orientation of $\omega$. This will be denoted by $\Zh^{\text{op}}(D)$. Up to semi-welded equivalence, $\Zh^{\text{op}}(D)$ is itself Alexander numerable. Conversely, if $D\cup \gamma$ is any Alexander numerable diagram which satisfies $(1)$, then $D \cup \gamma$ determines an Alexander system for $D$.  Our main theorem is below:

\begin{thm} \label{thm_main} Let $D_0,D_1$ be equivalent virtual link diagrams. Any Alexander system for $D_1$ is semi-welded equivalent to $\Zh^{\text{op}}(D_0)$. 
\end{thm}

In other words, the $\Zh^{\text{op}}$-construction turns a virtual link of $n$ components into an $(n+1)$ component almost classical link. Any other method of doing so by adding an over-crossing component is semi-welded equivalent to the $\Zh^{\text{op}}$-construction. 
\newline
\newline
As is well known, every almost classical link has a homologically trivial representative in some thickened surface $\Sigma \times \mathbb{I}$, where $\mathbb{I}=[0,1]$ and $\Sigma$ is closed and oriented. Hence $\mathscr{L}$ bounds a Seifert surface in $\Sigma \times \mathbb{I}$. Theorem \ref{thm_main} implies that the $\Zh^{\text{op}}$-construction can also be described geometrically in terms of Seifert surfaces. Suppose that a virtual link $J$ is represented by some link $\mathscr{J}\subset \Sigma \times \mathbb{I}$ and $F$ is a compact oriented surface in $\Sigma \times \mathbb{I}$ such that $\partial F=\mathscr{J} \cup \mathscr{C}$, where $\mathscr{C}$ is a collection of pairwise disjoint simple closed curves in $\Sigma \times \{1\}$. We show that such a cobounding surface $F$ for $\mathscr{L}$ exists and any given cobounding surface $F$ determines an Alexander system. In particular, we give an explicit $F$ which realizes the $\Zh^{\text{op}}$-construction. 
\newline
\newline
If $D$ itself is an almost classical link, it was proved in \cite{bbc} that $\Zh(D)$ is semi-welded equivalent to a split virtual link $D \sqcup \bigcirc$, where $\bigcirc$ is an unknot diagram disjoint from $D$. Using our model of the $\Zh$-construction, we give a simple geometric proof of this fact and its converse: a virtual link $D$ is almost classical if and only if $\Zh^{\text{op}}(D)$ is semi-welded equivalent to a split Alexander system $D \sqcup \bigcirc$. In this sense, Alexander systems detect almost classical links.  

\subsection{Main results: unification $\&$ extension} Thus far, the $\Zh$-construction has only been applied to studying extensions of the group of a virtual link. From this we obtain the generalized Alexander polynomial, the various index polynomials, and the extended Milnor invariants. Here we will consider invariants arising from the quantum point of view. That is, invariants arising from skein theory and solutions to the set-theoretic Yang-Baxter equation. 
\newline
\newline
In the inaugural paper on virtual knot theory \cite{KaV}, Kauffman showed that the skein relation for the bracket polynomial defines a virtual knot invariant that extends the Jones polynomial. Later, an inequivalent (and much more powerful) extension of the Jones polynomial was independently discovered by Dye-Kauffman \cite{dye2009virtual} and Miyazawa \cite{miyazawa2008}. Rather than taking values in $\mathbb{Z}[A^{\pm 1}]$, the Dye-Kauffman-Miyazawa polynomial is valued in $\mathbb{Z}[A^{\pm 1},K_1,K_2,\ldots]$. It is defined using a different skein relation than the classical case, but agrees with the usual bracket polynomial on classical knots. Here we give an alternate derivation of the DKM-polynomial using the $\Zh$-construction. The derivation uses only the usual Kauffman bracket skein relation. Virtual linking numbers of the $\omega$ component with the state curves are used to keep track of the variables $K_1,K_2,\ldots$. Thus, the $\Zh^{\text{op}}$-construction again brings virtual link invariants back into the realm of classical link theory.  
\newline
\newline
As a self-distributive structure, every quandle yields a solution to the set-theoretic Yang-Baxter equation. The usual quandle coloring and 2-cocycle invariants of classical links also give invariants of virtual links. We will show that any of these invariants can be extended using the $\Zh$-construction. The method is simply to take the fundamental quandle of $\Zh(D)$ and then use the usual coloring or $2$-cocyle invariants of $\Zh(D)$. For a fixed finite quandle, we give examples where the extended invariants are stronger than the usual ones.

\subsection{Organization} Section \ref{sec_back} reviews the elements of virtual knots, almost classical knots, and the $\Zh$-construction that are needed in the sequel. Alexander systems are defined and explored in Section \ref{sec_alex_system}. Theorem \ref{thm_main} is proved in Section \ref{sec_thm_main}. Alexander systems and coboundary surfaces are discussed in Section \ref{sec_geom_origin}. The derivation of the DKM-polynomial appears in Section \ref{AMpoly}. For extended quandles, see Section \ref{sec_ext_quandle}. Section \ref{sec_further} concludes the paper with some further discussion and questions for future research.

\begin{figure}[htb]
\begin{tabular}{ccc}
\begin{tabular}{c} \def\svgwidth{.5in}
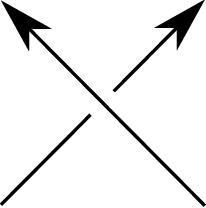 \\ negative \end{tabular} & \begin{tabular}{c} \def\svgwidth{.5in}
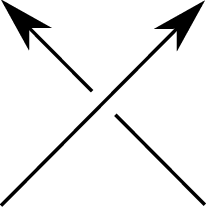 \\ positive \end{tabular} & \begin{tabular}{c} \def\svgwidth{.5in}
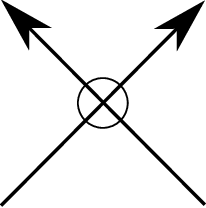 \\ virtual \end{tabular}
\end{tabular}
\caption{Crossings in virtual link diagrams.} \label{fig_crossings}
\end{figure}

\section{Background} \label{sec_back}

\subsection{Virtual knots} An (oriented) \emph{virtual knot diagram} is an (oriented) immersed circle on $\mathbb{R}^2$ such that each point of self-intersection is a transversal double point. Each double point is marked as either a positive classical crossing, a negative classical crossing, or a virtual crossing. See Figure \ref{fig_crossings}. In addition to planar isotopies, two types of moves can be performed on virtual knot diagrams (see Figure \ref{fig_moves}). First, the usual Reidemeister moves are allowed. Secondly, the detour move allows for any arc containing only virtual crossings to be redrawn in any position, where any new self-intersections are marked as virtual. The smallest equivalence relation on virtual knot diagrams generated by planar isotopies, Reidemeister moves, and detour moves will be called \emph{virtual Reidemeister equivalence}. If two virtual knot diagrams $D_0,D_1$ are virtually Reidemeister equivalent, we write $D_0 \leftrightharpoons D_1$. A virtual Reidemeister equivalence class will be called a \emph{virtual knot type}.
\newline

\begin{figure}[htb]
\begin{tabular}{cccc}
\begin{tabular}{c} \def\svgwidth{1in}
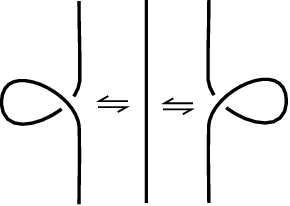 \end{tabular} & \begin{tabular}{c} \def\svgwidth{.66in}
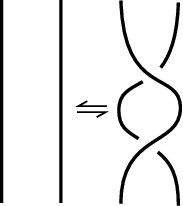  \end{tabular} & \begin{tabular}{c} \def\svgwidth{1.75in}
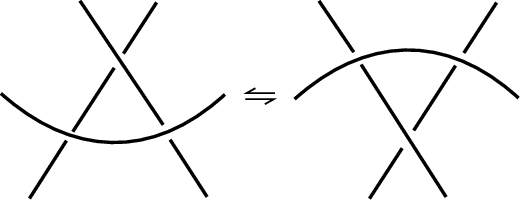  \end{tabular} & \begin{tabular}{c}
\def\svgwidth{1.8in}
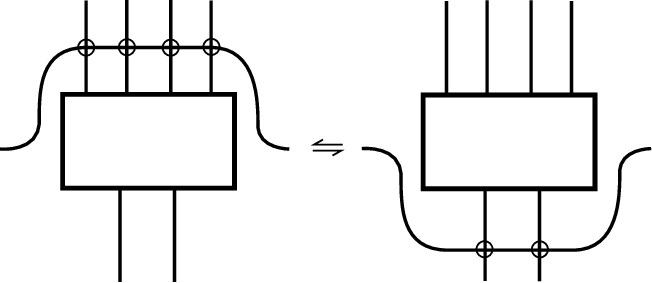
\end{tabular} \\
 $\Omega 1$ & $\Omega 2$ & $\Omega 3$ & Detour Move Example
\end{tabular}
\caption{The Reidemeister ($\Omega 1$, $\Omega 2$, $\Omega 3$) and detour moves.}\label{fig_moves}
\end{figure}

The \emph{Gauss diagram} $G$ of a virtual knot diagram $D$ is obtained as follows. For each classical crossing of $D$, connect the preimages of the corresponding double-point in the underlying immersion $\mathbb{S}^1 \to \mathbb{R}^2$ by a signed and directed chord of $\mathbb{S}^1$. The signed chord is called an \emph{arrow} of $G$. The arrow points from the over-crossing arc to the under-crossing arc. The sign of each arrow is the crossing sign: $\oplus$ for positive crossings and $\ominus$ for negative crossings. An example is shown in Figure \ref{fig_carter}. Note that virtual crossings are ignored. Gauss diagrams are considered equivalent up to orientation preserving diffeomorphisms of $\mathbb{S}^1$ that preserve the arrow endpoints, signs, and directions. The Reidemeister moves of Figure \ref{fig_moves} can be translated into moves on Gauss diagrams (see \cite{GPV}, Figure 6). The Reidemeister equivalence classes of Gauss diagrams are in one-to-one correspondence with the set of virtual knot types.
\newline
\newline
Let $n \ge 1$ be a natural number. A \emph{virtual link diagram of $n$ components} is an immersion of $n$ circles $\bigsqcup_{i=1}^n \mathbb{S}^1$ in the plane whose transversal double points are marked as either classical of virtual. A \emph{virtual link type} is a virtual Reidemeister equivalence class of virtual link diagrams. If $L$ is a virtual link diagram $L$ of $n$ components, then a \emph{Gauss diagram} $G$ of $L$ is formed by connecting each pair of preimages in each classical crossing of the immersion $\bigsqcup_{i=1}^n \mathbb{S}^1 \to \mathbb{R}^2$ by a signed and directed arrow. Note that the endpoints of an arrow may be in either the same or different component circles.
\newline

\begin{figure}[htb]
\begin{tabular}{c}
   \def\svgwidth{4in}
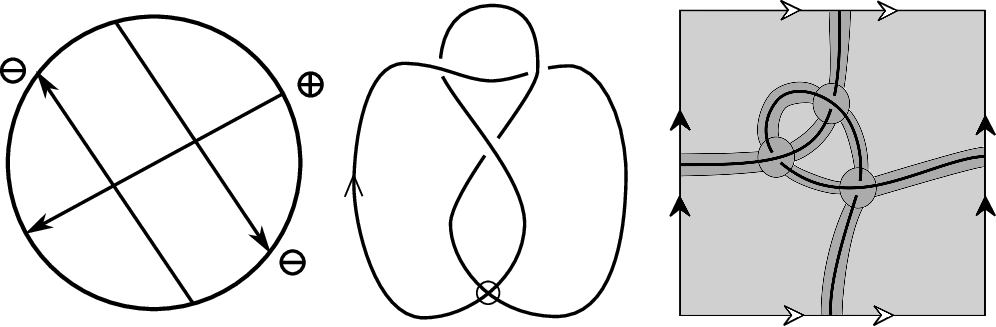 
\end{tabular}
\caption{A virtual knot diagram $K$ (center), a Gauss diagram $G$ of $K$ (left), and $K$ realized by a knot diagram $\mathscr{K}$ on its Carter surface $\Sigma \approx \mathbb{S}^1 \times \mathbb{S}^1$ (right).} \label{fig_carter}
\end{figure}

Every virtual link diagram $L$ can be drawn as a link diagram $\mathscr{L}$ on a closed orientable surface $\Sigma$. The surface $\Sigma$ is constructed as follows. Place a $0$-handle at each classical crossing of $L$. An untwisted $1$-handle is placed along each arc of $L$, where virtual crossings are ignored. Lastly, we attach $2$-handles to the boundary components to obtain the closed surface $\Sigma$. This is often called the \emph{Carter surface} of $L$. Conversely, given a link diagram $\mathscr{L}$ on a closed oriented surface $\Sigma$, a Gauss diagram $G$ of $\mathscr{L}$ is obtained in the same fashion as for virtual link diagrams. Hence, there are many surfaces on which a given virtual link diagram can be represented. In \cite{CKS}, Carter, Kamada, and Saito proved that there is a one-to-one correspondence between virtual link types and equivalence classes of knot diagrams on surfaces up to Reidemeister moves, diffeomorphisms of surfaces, and addition/subtraction of $1$-handles from surfaces that are disjoint from the link diagram.

\subsection{Almost classical links} \label{sec_review_ac} Let $D$ be a virtual link diagram. An \emph{arc} of $D$ is a path along $D$ from one under-crossing to the next, where virtual crossings are ignored. A \emph{short arc} of $D$ is a path along $D$ from one classical crossing to the next, where virtual crossings are again ignored. The short arcs of $D$ are in one-to-one correspondence to the arcs between adjacent arrow endpoints in a Gauss diagram of $D$. Let $S$ denote the set of short arcs of $D$. An \emph{Alexander numbering} of $D$ is a function $f:S \to \mathbb{Z}$ such that at each classical crossing of $D$, the assignment of short arcs to integers follows the rule depicted in Figure \ref{fig_alex_num_defn}. A virtual link diagram having an Alexander numbering is said to be \emph{Alexander numerable}.

\begin{definition}[Almost classical link]  A virtual link type $L$ is called \emph{almost classical} if there is an Alexander numerable virtual link diagram $D$ of $L$.
\label{acl_def}
\end{definition}

\begin{figure}[htb]
\begin{tabular}{cc}
\begin{tabular}{c}
\includegraphics[scale=.5]{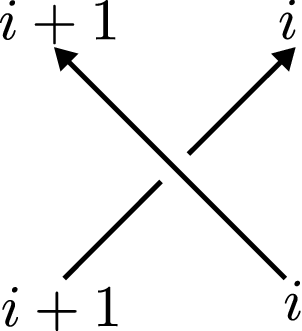} \end{tabular} & \begin{tabular}{c} \includegraphics[scale=.5]{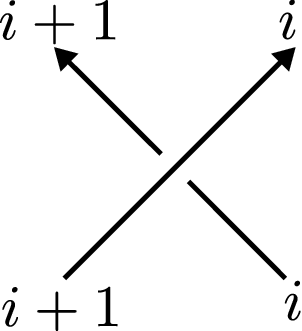} \end{tabular} 
\end{tabular}
\caption{An Alexander numbering at a crossing.} \label{fig_alex_num_defn}
\end{figure}

Almost classical links were first defined and studied by Silver and Williams \cite{silwill0,silwill1}. Note that not every diagram of an almost classical link is Alexander numerable. Given an Alexander numerable diagram $D$, let $\Sigma$ be its Carter surface, and $\mathscr{L}$ be the corresponding link in $\Sigma \times \mathbb{I}$. Then it follows that $\mathscr{L}$ is homologically trivial in $H_1(\Sigma \times \mathbb{I};\mathbb{Z})$. Hence, $\mathscr{L}$ bounds a Seifert surface in $\Sigma \times \mathbb{I}$. An algorithm for drawing such Seifert surfaces was given in \cite{vss}.
\newline
\newline
As every classical knot type can be represented by a homologically trivial knot in $\mathbb{S}^2 \times \mathbb{I}$, every classical knot type is almost classical. However, not every almost classical knot type is classical. In Green's table of virtual knots \cite{green}, only 77 of the 92800 virtual knot types up to six classical crossings are almost classical. Of these 77, exactly 10 are classical.

\subsection{The $\Zh$-construction} \label{sec_zh_defn} Here we will review the $\Zh$-construction and its basic properties. Further details can be found in \cite{dror,bbc,c_ext}. Recall that a \emph{welded knot type} is an equivalence class of virtual knot diagrams up to planar isotopies, Reidemeister moves, detour moves, and the over-crossings-commute move (OCC move) depicted in the left panel of Figure \ref{sw_equiv}. The $\Zh$-construction is well-defined up to \emph{semi-welded equivalence}, which we now describe.

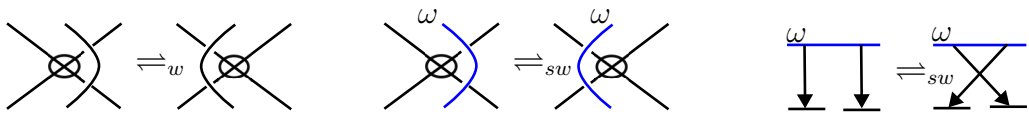
\begin{figure}[htb]
\begin{tabular}{c c c}
\begingroup%
  \makeatletter%
  \providecommand\color[2][]{%
    \errmessage{(Inkscape) Color is used for the text in Inkscape, but the package 'color.sty' is not loaded}%
    \renewcommand\color[2][]{}%
  }%
  \providecommand\transparent[1]{%
    \errmessage{(Inkscape) Transparency is used (non-zero) for the text in Inkscape, but the package 'transparent.sty' is not loaded}%
    \renewcommand\transparent[1]{}%
  }%
  \providecommand\rotatebox[2]{#2}%
  \newcommand*\fsize{\dimexpr\f@size pt\relax}%
  \newcommand*\lineheight[1]{\fontsize{\fsize}{#1\fsize}\selectfont}%
  \ifx\svgwidth\undefined%
    \setlength{\unitlength}{108.3734227bp}%
    \ifx\svgscale\undefined%
      \relax%
    \else%
      \setlength{\unitlength}{\unitlength * \real{\svgscale}}%
    \fi%
  \else%
    \setlength{\unitlength}{\svgwidth}%
  \fi%
  \global\let\svgwidth\undefined%
  \global\let\svgscale\undefined%
  \makeatother%
  \begin{picture}(1,0.34257981)%
    \lineheight{1}%
    \setlength\tabcolsep{0pt}%
    \put(0,0){\includegraphics[width=\unitlength]{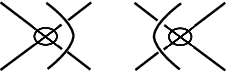}}%
    \put(0.45,0.14791085){\color[rgb]{0,0,0}\makebox(0,0)[lt]{\lineheight{1.25}\smash{\begin{tabular}[t]{l}$\rightleftharpoons_{w}$\end{tabular}}}}%
  \end{picture}%
\endgroup%
\hspace{30pt}
\begingroup%
  \makeatletter%
  \providecommand\color[2][]{%
    \errmessage{(Inkscape) Color is used for the text in Inkscape, but the package 'color.sty' is not loaded}%
    \renewcommand\color[2][]{}%
  }%
  \providecommand\transparent[1]{%
    \errmessage{(Inkscape) Transparency is used (non-zero) for the text in Inkscape, but the package 'transparent.sty' is not loaded}%
    \renewcommand\transparent[1]{}%
  }%
  \providecommand\rotatebox[2]{#2}%
  \newcommand*\fsize{\dimexpr\f@size pt\relax}%
  \newcommand*\lineheight[1]{\fontsize{\fsize}{#1\fsize}\selectfont}%
  \ifx\svgwidth\undefined%
    \setlength{\unitlength}{108.3734227bp}%
    \ifx\svgscale\undefined%
      \relax%
    \else%
      \setlength{\unitlength}{\unitlength * \real{\svgscale}}%
    \fi%
  \else%
    \setlength{\unitlength}{\svgwidth}%
  \fi%
  \global\let\svgwidth\undefined%
  \global\let\svgscale\undefined%
  \makeatother%
  \begin{picture}(1,0.34257981)%
    \lineheight{1}%
    \setlength\tabcolsep{0pt}%
    \put(0,0){\includegraphics[width=\unitlength]{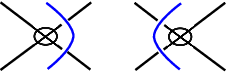}}%
    \put(0.12,0.31267558){\color[rgb]{0,0,0}\makebox(0,0)[lt]{\lineheight{1.25}\smash{\begin{tabular}[t]{l}$\omega$\end{tabular}}}}%
    \put(0.72,0.31162329){\color[rgb]{0,0,0}\makebox(0,0)[lt]{\lineheight{1.25}\smash{\begin{tabular}[t]{l}$\omega$\end{tabular}}}}%
    \put(0.45,0.14791085){\color[rgb]{0,0,0}\makebox(0,0)[lt]{\lineheight{1.25}\smash{\begin{tabular}[t]{l}$\rightleftharpoons_{sw}$\end{tabular}}}}%
  \end{picture}%
\endgroup%
 & \hspace{30pt}
\begingroup%
  \makeatletter%
  \providecommand\color[2][]{%
    \errmessage{(Inkscape) Color is used for the text in Inkscape, but the package 'color.sty' is not loaded}%
    \renewcommand\color[2][]{}%
  }%
  \providecommand\transparent[1]{%
    \errmessage{(Inkscape) Transparency is used (non-zero) for the text in Inkscape, but the package 'transparent.sty' is not loaded}%
    \renewcommand\transparent[1]{}%
  }%
  \providecommand\rotatebox[2]{#2}%
  \newcommand*\fsize{\dimexpr\f@size pt\relax}%
  \newcommand*\lineheight[1]{\fontsize{\fsize}{#1\fsize}\selectfont}%
  \ifx\svgwidth\undefined%
    \setlength{\unitlength}{90.70833702bp}%
    \ifx\svgscale\undefined%
      \relax%
    \else%
      \setlength{\unitlength}{\unitlength * \real{\svgscale}}%
    \fi%
  \else%
    \setlength{\unitlength}{\svgwidth}%
  \fi%
  \global\let\svgwidth\undefined%
  \global\let\svgscale\undefined%
  \makeatother%
  \begin{picture}(1,0.34688943)%
    \lineheight{1}%
    \setlength\tabcolsep{0pt}%
    \put(0,0){\includegraphics[width=\unitlength]{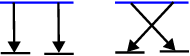}}%
    \put(-0.00256435,0.29480062){\color[rgb]{0,0,0}\makebox(0,0)[lt]{\lineheight{1.25}\smash{\begin{tabular}[t]{l}$\omega$\end{tabular}}}}%
    \put(0.60166132,0.29935069){\color[rgb]{0,0,0}\makebox(0,0)[lt]{\lineheight{1.25}\smash{\begin{tabular}[t]{l}$\omega$\end{tabular}}}}%
    \put(0.45,0.14091581){\color[rgb]{0,0,0}\makebox(0,0)[lt]{\lineheight{1.25}\smash{\begin{tabular}[t]{l}$\rightleftharpoons_{sw}$\end{tabular}}}}%
  \end{picture}%
\endgroup%
 
\end{tabular}

\caption{Welded equivalence: over-crossings commute (left). Semi-welded equivalence: $\omega$-over-crossings commute (middle). Gauss diagram formulation of semi-welded equivalence (right).}
\label{sw_equiv}
\end{figure}

\begin{definition}[Semi-welded equivalence] Let $L,L'$ be two $(n+1)$ component virtual link diagrams with components labeled $1,2,\ldots,n,n+1$. The $(n+1)$-st component will be called the $\omega$ component and will be colored blue in figures. We say that $L,L'$  are \emph{semi-welded equivalent} if they may be obtained from one another by a finite sequence of Reidemeister moves, detour moves, and  $\omega$-over-crossings-commute ($\omega$OCC) moves (see Figure \ref{sw_equiv} middle and right) in which the over-crossing arc is part of the $\omega$ component. If $L,L'$ are semi-welded equivalent, we write $L \leftrightharpoons_{sw} L'$.
\end{definition}

\begin{definition}[$\Zh$-construction]
Let $D$ be an $n$ component link diagram. Let $\Zh(D)$ be the $(n+1)$ component diagram constructed by adding a new component, denoted $\omega$, as follows. Add two solid blue over-crossing arcs around each classical crossing in the diagram $D$ as depicted in Figure \ref{zh_construction}. Then connect the ends of the arcs into a single component. Connections may be made in order. To emphasize the arbitrary nature of these connections, the connecting arcs are drawn as dashed. New crossings of $\omega$ with $D$ and $\omega$ with itself are marked as virtual crossings. We will refer to $D$ as the base component(s) of $\Zh(D)$.  
\end{definition}

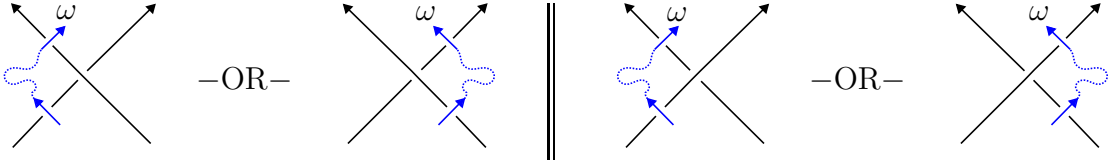
\begin{figure}[htb]
\begin{tabular}{cccc||cccc}
\begin{tabular}{c} \def\svgwidth{.8in}
\begingroup%
  \makeatletter%
  \providecommand\color[2][]{%
    \errmessage{(Inkscape) Color is used for the text in Inkscape, but the package 'color.sty' is not loaded}%
    \renewcommand\color[2][]{}%
  }%
  \providecommand\transparent[1]{%
    \errmessage{(Inkscape) Transparency is used (non-zero) for the text in Inkscape, but the package 'transparent.sty' is not loaded}%
    \renewcommand\transparent[1]{}%
  }%
  \providecommand\rotatebox[2]{#2}%
  \newcommand*\fsize{\dimexpr\f@size pt\relax}%
  \newcommand*\lineheight[1]{\fontsize{\fsize}{#1\fsize}\selectfont}%
  \ifx\svgwidth\undefined%
    \setlength{\unitlength}{142.47055229bp}%
    \ifx\svgscale\undefined%
      \relax%
    \else%
      \setlength{\unitlength}{\unitlength * \real{\svgscale}}%
    \fi%
  \else%
    \setlength{\unitlength}{\svgwidth}%
  \fi%
  \global\let\svgwidth\undefined%
  \global\let\svgscale\undefined%
  \makeatother%
  \begin{picture}(1,0.95172102)%
    \lineheight{1}%
    \setlength\tabcolsep{0pt}%
    \put(0,0){\includegraphics[width=\unitlength]{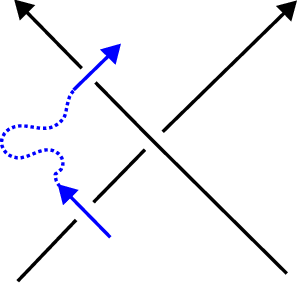}}%
    \put(0.34429434,0.8459713){\color[rgb]{0,0,0}\makebox(0,0)[lt]{\lineheight{40.54999924}\smash{\begin{tabular}[t]{l}$\omega$\end{tabular}}}}%
  \end{picture}%
\endgroup%
 \end{tabular} & $-$OR$-$ & 
\begin{tabular}{c} \def\svgwidth{.8in}
\begingroup%
  \makeatletter%
  \providecommand\color[2][]{%
    \errmessage{(Inkscape) Color is used for the text in Inkscape, but the package 'color.sty' is not loaded}%
    \renewcommand\color[2][]{}%
  }%
  \providecommand\transparent[1]{%
    \errmessage{(Inkscape) Transparency is used (non-zero) for the text in Inkscape, but the package 'transparent.sty' is not loaded}%
    \renewcommand\transparent[1]{}%
  }%
  \providecommand\rotatebox[2]{#2}%
  \newcommand*\fsize{\dimexpr\f@size pt\relax}%
  \newcommand*\lineheight[1]{\fontsize{\fsize}{#1\fsize}\selectfont}%
  \ifx\svgwidth\undefined%
    \setlength{\unitlength}{142.47055229bp}%
    \ifx\svgscale\undefined%
      \relax%
    \else%
      \setlength{\unitlength}{\unitlength * \real{\svgscale}}%
    \fi%
  \else%
    \setlength{\unitlength}{\svgwidth}%
  \fi%
  \global\let\svgwidth\undefined%
  \global\let\svgscale\undefined%
  \makeatother%
  \begin{picture}(1,0.95172102)%
    \lineheight{1}%
    \setlength\tabcolsep{0pt}%
    \put(0,0){\includegraphics[width=\unitlength]{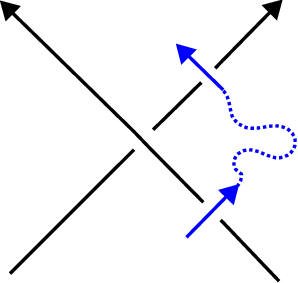}}%
    \put(0.48072041,0.84197892){\color[rgb]{0,0,0}\makebox(0,0)[lt]{\lineheight{40.54999924}\smash{\begin{tabular}[t]{l}$\omega$\end{tabular}}}}%
  \end{picture}%
\endgroup%
 \end{tabular} & & & \begin{tabular}{c} \def\svgwidth{.8in}
\begingroup%
  \makeatletter%
  \providecommand\color[2][]{%
    \errmessage{(Inkscape) Color is used for the text in Inkscape, but the package 'color.sty' is not loaded}%
    \renewcommand\color[2][]{}%
  }%
  \providecommand\transparent[1]{%
    \errmessage{(Inkscape) Transparency is used (non-zero) for the text in Inkscape, but the package 'transparent.sty' is not loaded}%
    \renewcommand\transparent[1]{}%
  }%
  \providecommand\rotatebox[2]{#2}%
  \newcommand*\fsize{\dimexpr\f@size pt\relax}%
  \newcommand*\lineheight[1]{\fontsize{\fsize}{#1\fsize}\selectfont}%
  \ifx\svgwidth\undefined%
    \setlength{\unitlength}{142.47055229bp}%
    \ifx\svgscale\undefined%
      \relax%
    \else%
      \setlength{\unitlength}{\unitlength * \real{\svgscale}}%
    \fi%
  \else%
    \setlength{\unitlength}{\svgwidth}%
  \fi%
  \global\let\svgwidth\undefined%
  \global\let\svgscale\undefined%
  \makeatother%
  \begin{picture}(1,0.95172102)%
    \lineheight{1}%
    \setlength\tabcolsep{0pt}%
    \put(0,0){\includegraphics[width=\unitlength]{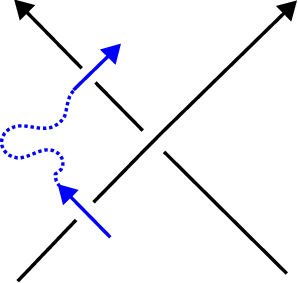}}%
    \put(0.33024771,0.83192452){\color[rgb]{0,0,0}\makebox(0,0)[lt]{\lineheight{40.54999924}\smash{\begin{tabular}[t]{l}$\omega$\end{tabular}}}}%
  \end{picture}%
\endgroup%
 \end{tabular} & $-$OR$-$ &
\begin{tabular}{c} \def\svgwidth{.8in}
\begingroup%
  \makeatletter%
  \providecommand\color[2][]{%
    \errmessage{(Inkscape) Color is used for the text in Inkscape, but the package 'color.sty' is not loaded}%
    \renewcommand\color[2][]{}%
  }%
  \providecommand\transparent[1]{%
    \errmessage{(Inkscape) Transparency is used (non-zero) for the text in Inkscape, but the package 'transparent.sty' is not loaded}%
    \renewcommand\transparent[1]{}%
  }%
  \providecommand\rotatebox[2]{#2}%
  \newcommand*\fsize{\dimexpr\f@size pt\relax}%
  \newcommand*\lineheight[1]{\fontsize{\fsize}{#1\fsize}\selectfont}%
  \ifx\svgwidth\undefined%
    \setlength{\unitlength}{142.47055229bp}%
    \ifx\svgscale\undefined%
      \relax%
    \else%
      \setlength{\unitlength}{\unitlength * \real{\svgscale}}%
    \fi%
  \else%
    \setlength{\unitlength}{\svgwidth}%
  \fi%
  \global\let\svgwidth\undefined%
  \global\let\svgscale\undefined%
  \makeatother%
  \begin{picture}(1,0.95172102)%
    \lineheight{1}%
    \setlength\tabcolsep{0pt}%
    \put(0,0){\includegraphics[width=\unitlength]{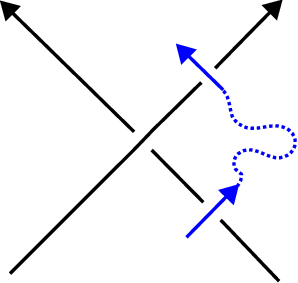}}%
    \put(0.48072041,0.84197892){\color[rgb]{0,0,0}\makebox(0,0)[lt]{\lineheight{40.54999924}\smash{\begin{tabular}[t]{l}$\omega$\end{tabular}}}}%
  \end{picture}%
\endgroup%
 \end{tabular}
\end{tabular}
\caption{Drawing the $\omega$ component in $\Zh(D)$. Each positive and negative classical crossing of $D$ contributes to $\omega$ in one of the two depicted ways.}
\label{zh_construction}
\end{figure}

\begin{example} Figure \ref{fig_zh_example} shows one way to draw the $\omega$ component for the positive right-handed trefoil. Any two ways of drawing the $\omega$ component are semi-welded equivalent.
    
\end{example}

\begin{figure}[htb]
\tiny
\[
\xymatrix{\begin{array}{c} \includegraphics[scale=.6]{zh_virt_tref.eps} \end{array} \ar[r]^-{\Zh} & \begin{array}{c} 
\includegraphics[scale=.6]{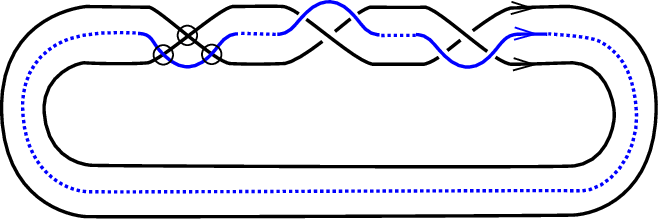} \end{array}} 
\]
\normalsize
\caption{The $\Zh$-construction for the positive virtual trefoil.}
\label{fig_zh_example}
\end{figure}

\begin{theorem}[\cite{bbc}, Proposition 4.2] The $\Zh$-construction is well defined up to semi-welded equivalence.
Furthermore, If $D_{0}$ and $D_{1}$ are equivalent virtual link diagrams then $\Zh(D_{0})$
and $\Zh(D_{0})$ are semi-welded equivalent. 
\label{sw_eq_thrm}
\end{theorem}

It is important to emphasize the role of semi-welded equivalence in the definition of the $\Zh$ map. First note that if $\omega,\omega'$ are any two ways of connecting the  solid blue arcs together for a virtual link diagram $D$, then $D \cup \omega \leftrightharpoons_{sw} D \cup \omega'$. To see this, choose a base point on the $\omega$ component. Using detour moves, it can be arranged so that when traveling along $\omega$, all of the virtual crossings of $\omega$ with $D$ are passed before the classical crossings. Then passing a pair of classical crossings of $\omega$ with $D$ appears as in Figure \ref{sw_AlexNum}, left. Using a detour move and an $\omega$OCC move, the order of two adjacent crossings can be transposed. See again Figure \ref{sw_AlexNum}. Since all permutations are generated by transpositions, it follows that the over-crossings of $\omega$ with $D$ can be arranged in any order. Hence, the $\Zh$-construction is well defined.  
\newline
\begin{figure}[htb]
\[
\begin{array}{c} \includegraphics[scale=.35]{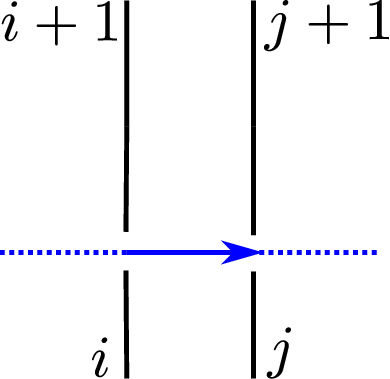} \end{array} \leftrightharpoons_{sw} \begin{array}{c} \includegraphics[scale=.35]{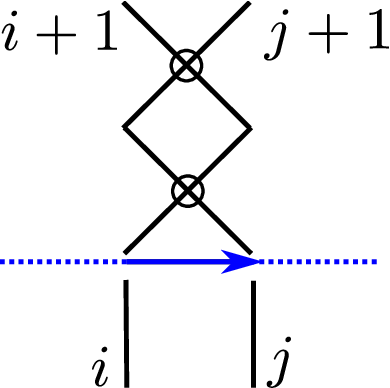} \end{array} \leftrightharpoons_{sw} \begin{array}{c} \includegraphics[scale=.35]{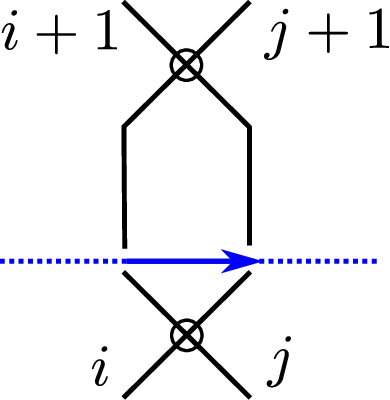} \end{array} 
\]
\caption{The $\omega \text{OCC}$ move allows the order of the classical crossings along $\omega$ to be switched without changing the integer labeling of the base component.}
\label{sw_AlexNum}
\end{figure}

While the recipe for constructing $\Zh(D)$ is completely diagrammatic, it encodes interesting geometric properties of the virtual link type of $D$. For example, the $\Zh$-map is functorial under virtual link concordance (see \cite{bbc}, Theorem 4.3). In \cite{dror}, Bar-Natan noted that if $D$ is classical, then $\Zh(D)$ is semi-welded equivalent to a split $2$-component classical link diagram $D \cup \omega$ where $\omega$ is an unknotted circle. In \cite{bbc}, this was generalized to the following result for almost classical links. A geometric proof of this theorem will be in Section \ref{sec_geom_origin} ahead. 

\begin{theorem}[\cite{bbc}, Theorem 4.5] \label{thm_splits} If $L$ is a diagram of an almost classical link, then $\Zh(L)$ is semi-welded equivalent to the split link $L \sqcup \omega$, where $\omega$ is an unknot disjoint from the diagram of $L$.
\end{theorem}
 
The $\Zh$-construction has been defined here following the orientation conventions of \cite{dror,bbc,c_ext}. For many of our applications, the opposite orientation of $\omega$ is more natural. 

\begin{definition}[$\Zh^{\text{op}}$-construction] By $\Zh^{\text{op}}(D)=D \cup \omega^{\text{op}}$, we mean a virtual link diagram obtained from $\Zh(D)=D\cup \omega$ by changing the orientation of $\omega$.
\end{definition}

Clearly, the $\Zh^{\text{op}}$-construction is well-defined up to semi-welded equivalence and if $D_0,D_1$ are equivalent virtual link diagrams, then $\Zh^{\text{op}}(D_0) \leftrightharpoons_{sw} \Zh^{\text{op}}(D_1)$. Furthermore, if $D$ is almost classical, then $\Zh^{\text{op}}(D)$ splits as in Theorem \ref{thm_splits}.

\section{Characterizing the \texorpdfstring{$\Zh$}{Lg}-construction} 

\subsection{Alexander systems} \label{sec_alex_system} Let $D$ be an $n$ component virtual link diagram and $D\cup \gamma$ an $(n+1)$ component virtual link containing $D$ as a sublink. In figures, the component $\gamma$ will be drawn in blue while the diagram $D$ will be drawn in black. Each short arc of $D \cup \gamma$ lies in a short arc of the original link $D$ or in the added component $\gamma$. Let $D_{\gamma}$ denote the subset of the short arcs of $D \cup \gamma$ that are contained in $D$. 

\begin{figure}[htb]
\begin{tabular}{ccc}
\begin{tabular}{c} \includegraphics[scale=.5]{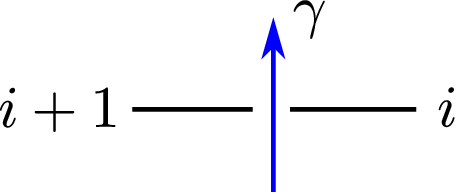} \end{tabular} & \begin{tabular}{c}
\includegraphics[scale=.4]{alex_numb_def_nice_II.eps} \end{tabular} & \begin{tabular}{c} \includegraphics[scale=.4]{alex_numb_def_nice.eps} \end{tabular} \\
Definition  \ref{defn_zh_system}, (2a) & \multicolumn{2}{c}{Definition  \ref{defn_zh_system}, (2b)} 
\end{tabular}
\caption{Alexander sub-numberings.} \label{fig_zh_system_defn}
\end{figure}

\begin{definition}[Alexander system] \label{defn_zh_system} Let $D$ be an $n$ component virtual link diagram. An \emph{Alexander system for $D$} is a triple $(D,\gamma,\Gamma)$, where $D \cup \gamma$ is an $(n+1)$ component virtual link diagram and $\Gamma:D_{\gamma} \to \mathbb{Z}$ is an integer labeling of $D_{\gamma}$ such that the following two conditions are satisfied.
\begin{enumerate}
    \item For every classical crossing involving the $\gamma$ component, the over-crossing arc lies in $\gamma$ and the under-crossing arcs lie in $D$.  
\item $\Gamma$ is an \emph{Alexander sub-numbering} of $D_{\gamma}$. That is, $\Gamma$ resembles:
\begin{enumerate}    
    \item Figure \ref{fig_zh_system_defn}, left, at every classical crossing of $\gamma$ and $D$, and     
    \item Figure \ref{fig_zh_system_defn}, right, at every classical self-crossing of $D$.
  \end{enumerate}
\end{enumerate}
Two Alexander systems $(D,\gamma,\Gamma)$,$(D,\gamma',\Gamma')$ for $D$ are called  \emph{equivalent} if $D \cup \gamma \leftrightharpoons_{sw} D \cup \gamma'$. 
\end{definition}

Some remarks about this definition are in order. First, observe that the labeling function $\Gamma$ is not used in the definition of equivalence for Alexander systems. If an Alexander sub-numbering for $D_{\gamma}$ exists, then the choice of this function can be made arbitrarily. It is fixed in the notation simply for the sake of definiteness. Secondly, it is important to emphasize that the short arcs of $\gamma$ in $D \cup \gamma$ are not labeled by $\Gamma$. Only the short arcs in $D_{\gamma}$ are labeled. The arcs of $\gamma$ are not labeled because the virtual link diagram $D \cup \gamma$ determined by an Alexander system $(D,\gamma,\Gamma)$ need not be Alexander numerable. Nonetheless, as the next lemma shows, the $\Zh^{\text{op}}$-construction can be drawn in such a way that in the resulting Alexander system $(D,\omega^{op},\Omega)$, $D \cup \omega^{op}$ is Alexander numerable.

\begin{figure}[htb]
\begin{subfigure}[b]{0.4\textwidth}
\centering
\includegraphics[scale=.4]{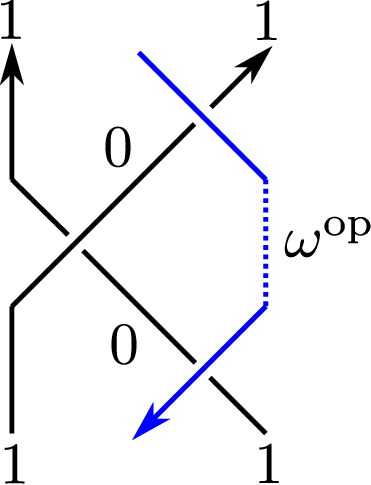}
\caption{}
\end{subfigure}
\begin{subfigure}[b]{0.4\textwidth}
\centering
\includegraphics[scale=.4]{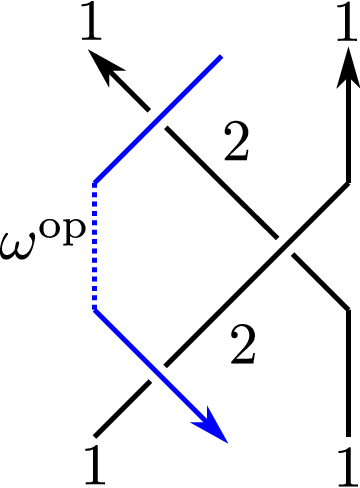}
\label{fig:AN_base_component}
\caption{}
\end{subfigure}
\caption{Panel (A) shows the Alexander sub-numbering with $\omega^{\text{op}}$ on the right. Panel (B) shows the Alexander sub-numbering for $\omega^{\text{op}}$ on the left. The same Alexander sub-numberings are used for negative crossings of $D$.} \label{zh-alex-sub}.
\end{figure}

\begin{figure}[htb]
\begin{subfigure}[b]{0.4\textwidth}
\centering
\includegraphics[scale=.4]{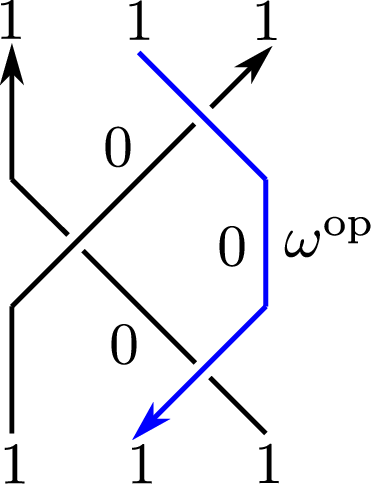}
\caption{}
\label{Zh_op_is_AN_right}
\end{subfigure}
\begin{subfigure}[b]{0.4\textwidth}
\centering
\includegraphics[scale=.4]{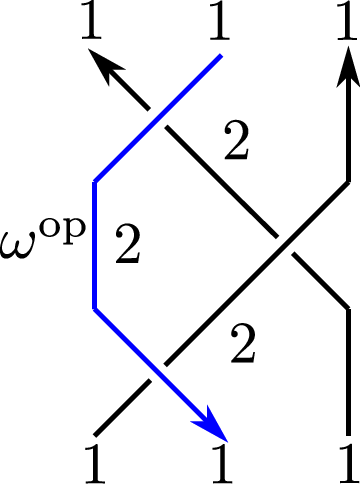}
\caption{}
\label{Zh_op_is_AN_left}
\end{subfigure}
\caption{Panel (A) shows the Alexander numbering with $\omega^{\text{op}}$ on the right. Panel (B) shows the Alexander numbering for $\omega^{\text{op}}$ on the left.  The same Alexander numberings are used for negative crossings of $D$.} \label{Zh_op_is_AN}.
\end{figure}

\begin{lemma} \label{lemma_zh_is_alexander} Let $D$ be a virtual link diagram. Then the following hold.
\begin{enumerate}
\item $\Zh^{\text{op}}(D)=D \cup \omega^{\text{op}}$ defines an Alexander system $(D,\omega^{\text{op}},\Omega)$ for $D$.
\item The $\omega^{\text{op}}$ component of $\Zh^{\text{op}}(D)$ can be drawn so that $D \cup \omega^{op}$ is Alexander numerable. 
\end{enumerate}
\end{lemma}
\begin{proof}
First suppose that near a positive crossing of $D$, $\Zh^{\text{op}}(D)$ resembles Figure \ref{zh-alex-sub}, left. Label the short arcs of $D_{\gamma}$ with $0$ or $1$ as shown in Figure \ref{zh-alex-sub}, left. Alternatively, the two blue arcs flanking a classical crossing of $D$ in $\Zh^{\text{op}}(D)$ can instead be drawn on the left, as in Figure \ref{zh-alex-sub}, right. In this case, the short arcs may be labeled as shown with a $1$ or $2$. For a negative crossing of $D$, the labeling of the short arcs is obtained from Figure \ref{zh-alex-sub} by changing the depicted crossing of $D$ from positive to negative. All together, this defines a function $\Omega:D_{\omega^{\text{op}}} \to \mathbb{Z}$ satisfying Definition \ref{defn_zh_system}. The labeling is consistent since all the incoming and outgoing short arcs near a classical crossing of $D$ are labeled with a $1$.
\newline

For the second claim, we begin at a classical crossing $x$ of $D$. Connect the two blue over-crossing arcs of $\omega^{\text{op}}$ near $x$ so that they appear as on the left or right in Figure \ref{Zh_op_is_AN}. Do this at every classical crossing $x$ of $D$. Then connect these longer pieces together to finish making $\omega^{\text{op}}$. Next, label the short arcs of $D \cup \omega^{\text{op}}$ as shown in Figure \ref{Zh_op_is_AN}. Although a positive crossing is shown, the same connections and labeling is to be used at negative crossings of $D$. This labeling is an Alexander numbering as all the incoming and outgoing arcs are labeled $1$.  \end{proof}

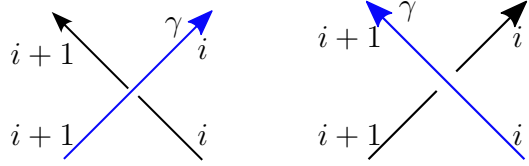
\begin{figure}[htb]
\def\svgwidth{.85 in} 
\begingroup%
  \makeatletter%
  \providecommand\color[2][]{%
    \errmessage{(Inkscape) Color is used for the text in Inkscape, but the package 'color.sty' is not loaded}%
    \renewcommand\color[2][]{}%
  }%
  \providecommand\transparent[1]{%
    \errmessage{(Inkscape) Transparency is used (non-zero) for the text in Inkscape, but the package 'transparent.sty' is not loaded}%
    \renewcommand\transparent[1]{}%
  }%
  \providecommand\rotatebox[2]{#2}%
  \newcommand*\fsize{\dimexpr\f@size pt\relax}%
  \newcommand*\lineheight[1]{\fontsize{\fsize}{#1\fsize}\selectfont}%
  \ifx\svgwidth\undefined%
    \setlength{\unitlength}{124.36707115bp}%
    \ifx\svgscale\undefined%
      \relax%
    \else%
      \setlength{\unitlength}{\unitlength * \real{\svgscale}}%
    \fi%
  \else%
    \setlength{\unitlength}{\svgwidth}%
  \fi%
  \global\let\svgwidth\undefined%
  \global\let\svgscale\undefined%
  \makeatother%
  \begin{picture}(1,0.62992838)%
    \lineheight{1}%
    \setlength\tabcolsep{0pt}%
    \put(0,0){\includegraphics[width=\unitlength]{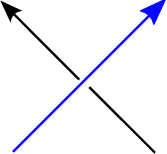}}%
	\put(0.7,0.8){\makebox(0,0)[lt]{\lineheight{1.25}\smash{\begin{tabular}[t]{l}$\gamma$\end{tabular}}}}%
    \put(0.9,0.63){\makebox(0,0)[lt]{\lineheight{1.25}\smash{\begin{tabular}[t]{l}$i$\end{tabular}}}}%
    \put(0.9,0.1){\makebox(0,0)[lt]{\lineheight{1.25}\smash{\begin{tabular}[t]{l}$i$\end{tabular}}}}%
    \put(-0.25,0.1){\makebox(0,0)[lt]{\lineheight{1.25}\smash{\begin{tabular}[t]{l}$i+1$\end{tabular}}}}%
    \put(-0.25,0.6){\makebox(0,0)[lt]{\lineheight{1.25}\smash{\begin{tabular}[t]{l}$i+1$\end{tabular}}}}%
  \end{picture}%
\endgroup%
 
\def\svgwidth{.85 in}\hspace{50pt}
\begingroup%
  \makeatletter%
  \providecommand\color[2][]{%
    \errmessage{(Inkscape) Color is used for the text in Inkscape, but the package 'color.sty' is not loaded}%
    \renewcommand\color[2][]{}%
  }%
  \providecommand\transparent[1]{%
    \errmessage{(Inkscape) Transparency is used (non-zero) for the text in Inkscape, but the package 'transparent.sty' is not loaded}%
    \renewcommand\transparent[1]{}%
  }%
  \providecommand\rotatebox[2]{#2}%
  \newcommand*\fsize{\dimexpr\f@size pt\relax}%
  \newcommand*\lineheight[1]{\fontsize{\fsize}{#1\fsize}\selectfont}%
  \ifx\svgwidth\undefined%
    \setlength{\unitlength}{74.08270828bp}%
    \ifx\svgscale\undefined%
      \relax%
    \else%
      \setlength{\unitlength}{\unitlength * \real{\svgscale}}%
    \fi%
  \else%
    \setlength{\unitlength}{\svgwidth}%
  \fi%
  \global\let\svgwidth\undefined%
  \global\let\svgscale\undefined%
  \makeatother%
  \begin{picture}(1,0.99133952)%
    \lineheight{1}%
    \setlength\tabcolsep{0pt}%
    \put(0,0){\includegraphics[width=\unitlength]{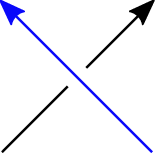}}%
	\put(0.2,0.9){\makebox(0,0)[lt]{\lineheight{1.25}\smash{\begin{tabular}[t]{l}$\gamma$\end{tabular}}}}%
    \put(0.9,0.7){\makebox(0,0)[lt]{\lineheight{1.25}\smash{\begin{tabular}[t]{l}$i$\end{tabular}}}}%
    \put(0.9,0.1){\makebox(0,0)[lt]{\lineheight{1.25}\smash{\begin{tabular}[t]{l}$i$\end{tabular}}}}%
    \put(-0.3,0.7){\makebox(0,0)[lt]{\lineheight{1.25}\smash{\begin{tabular}[t]{l}$i+1$\end{tabular}}}}%
    \put(-0.3,0.1){\makebox(0,0)[lt]{\lineheight{1.25}\smash{\begin{tabular}[t]{l}$i+1$\end{tabular}}}}%
  \end{picture}%
\endgroup%
 
\caption{By ignoring the label on $\gamma$, we see that the labeling on $D_{\gamma}$ satisfies property $(2a)$ of the definition of an Alexander system.}
\label{fig_alex_num_is_system}
\end{figure}

\begin{lemma} \label{lemma_alex_num_is_system} Suppose that $D$ is an $n$ component virtual link diagram and $D \cup \gamma$ is an $(n+1)$ component virtual link diagram satisfying condition $(1)$ of Definition \ref{defn_zh_system}. If $D \cup \gamma$ is Alexander numerable, then there is an Alexander system $(D,\gamma,\Gamma)$ of $D$.
\end{lemma}
\begin{proof} Let $f:S \to \mathbb{Z}$ be an Alexander numbering of the short arcs of $D \cup \gamma$. Let $\Gamma:D_{\gamma} \to \mathbb{Z}$ be the restriction of $S$ to the short arcs of $D$ in $D \cup \gamma$. Since $f$ is an Alexander numbering of $D \cup \gamma$, $\Gamma$ automatically satisfies condition $(2b)$ of Definition \ref{defn_zh_system}. By condition $(1)$, every classical crossing involving $\gamma$ resembles one of the two pictures shown in Figure \ref{fig_alex_num_is_system}. It follows from Figure \ref{fig_alex_num_is_system} that the labeling inherited from $f$ satisfies condition $(2a)$. Thus, $(D,\gamma,\Gamma)$ is an Alexander system.
\end{proof}
 
\begin{corollary} \label{cor_split_iff_alex_num} Let $D$ be an $n$ component virtual link diagram. Let $D \sqcup \bigcirc$ be the $(n+1)$ component diagram obtained from $D$ by drawing an unknotted circle disjoint from $D$. Then $D$ is Alexander numerable if and only if $D \sqcup \bigcirc$ gives an Alexander system $(D,\bigcirc,\Gamma)$.
\end{corollary}
\begin{proof} The diagram $D \sqcup \bigcirc$ vacuously satisfies condition $(1)$ of Definition \ref{defn_zh_system}. If $D \sqcup \bigcirc$ is Alexander numerable, then there is an Alexander system $(D,\bigcirc,\Gamma)$ by Lemma \ref{lemma_alex_num_is_system}. Conversely, if there is an Alexander system $(D,\bigcirc,\Gamma)$, then the absence of crossings between $\bigcirc$ and $D$, together with condition $(2b)$, implies that $D$ itself is Alexander numerable.
\end{proof} 
 
Lastly, we consider a very different type of Alexander system for a virtual link diagram $D$. First take every virtual crossing of $D$ and convert it to either a positive or negative classical crossing. The result is a classical link diagram and thus there is an Alexander numbering of its short arcs. Next, convert these altered  crossings back to virtual crossings. To preserve the Alexander numbering of the short arcs, add in over-crossing arcs near each virtual crossing as pictured in Figure \ref{fig_virt_alex_system}. Connect these arcs together in an arbitrary manner to create a single component $\upsilon$. As usual, any new intersections between $\upsilon$ and $D$ or $\upsilon$ and itself are marked as virtual crossings. With this labeling $\Upsilon$, the $(n+1)$ component virtual link diagram $D \cup \upsilon$ satisfies conditions $(1)$ and $(2)$ of Definition \ref{defn_zh_system}. This  Alexander system for $D$ will be used in Section \ref{AMpoly}. It is recorded as a definition below for future reference.

\begin{definition}[Virtual Alexander system] \label{defn_virtual_alex_system} An Alexander system $(D,\upsilon,\Upsilon)$ for $D$ constructed as in the preceding paragraph will be called a \emph{virtual Alexander system for $D$}.
\end{definition}

\begin{figure}[htb]
\centering
\[
\xymatrix{
\begin{array}{c}
\includegraphics[scale=.4]{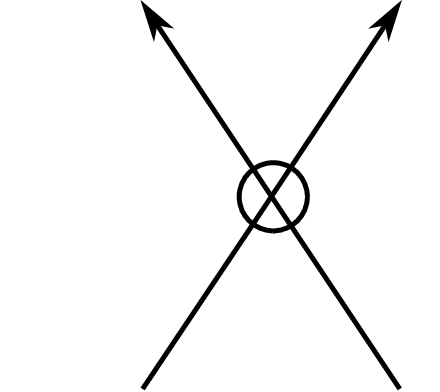}
\end{array} \ar[r] &
\begin{array}{c}
\includegraphics[scale=.4]{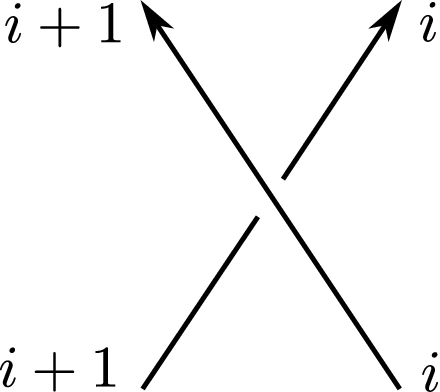}
\end{array} \ar[r] &
\begin{array}{c}
\includegraphics[scale=.4]{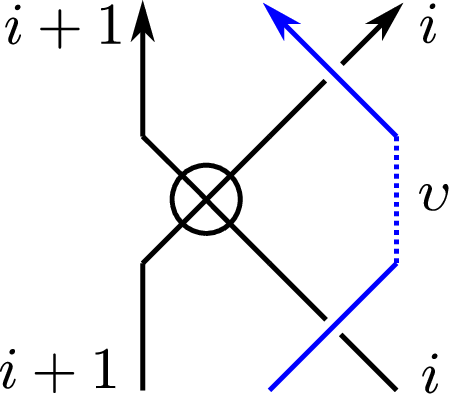}
\end{array}
}
\]
\caption{Construction of a virtual Alexander system. See Definition \ref{defn_virtual_alex_system}.}
\label{fig_virt_alex_system}
\end{figure}

\subsection{Proof of Theorem \ref{thm_main}} \label{sec_thm_main} Our proof generalizes the classification of cut-point systems due to N. Kamada \cite{kamada_cutpoint1}. We begin by establishing a collection of moves on Alexander systems that are needed in the proof.  First it is shown that reordering the over-crossing arcs of $\gamma$ in an Alexander system $(D,\gamma,\Gamma)$ gives an equivalent Alexander system for $D$ with the same labeling $\Gamma$. More precisely, suppose that $D \cup \gamma'$ is a virtual link diagram obtained from the given diagram $D \cup \gamma$ as follows. Erase all of $\gamma$ except for small over-crossing arcs at each of the classical crossings of $\gamma$ with $D$. Then reconnect the leftover pieces in an arbitrary fashion so that the result has one component $\gamma'$. All new intersections between $\gamma'$ and $D$, or $\gamma'$ and itself, are required to be transversal intersections that are disjoint from all other crossings. Mark all of these added crossings as virtual crossings. Note that in this case $D_{\gamma}=D_{\gamma'}$ and hence we can label the short arcs of $D$ in $D \cup \gamma'$ using the same function $\Gamma$.

\begin{lemma} \label{lemma_alex_and_sw} If $D \cup \gamma'$ is obtained from an Alexander system $(D, \gamma,\Gamma)$ by rearranging the over-crossing arcs of $\gamma$ with $D$ as above, then $(D,\gamma,\Gamma)$ is equivalent to $(D,\gamma',\Gamma)$. 
\end{lemma}

\begin{proof} Since $\gamma'$ and $\gamma$ have the same classical crossings with $D$ and $D_{\gamma}=D_{\gamma'}$, the labeling $f:D_{\gamma}=D_{\gamma'} \to \mathbb{Z}$ also gives a labeling of $D_{\gamma'}$ satisfying $(2)$ of Definition \ref{defn_zh_system}. Hence $(D,\gamma',\Gamma)$ is an Alexander system. It remains to show that $D \cup \gamma$ and $D\cup \gamma'$ are semi-welded equivalent. First observe that detour moves and the $\omega$OCC move do not change the labeling of $D_{\gamma'}$. This is obvious for detour moves. For the $\omega$OCC move, see Figure \ref{sw_AlexNum}. Now, following the proof that the $\Zh$-construction is well defined, any desired ordering of the under-crossing arcs of $D$ along $\gamma'$ can be achieved from the given one by permuting pairs of adjacent under-crossing arcs. Virtual crossings can likewise be placed in any desired position using detour moves. Hence, $(D,\gamma,\Gamma)$ and $(D,\gamma',\Gamma')$ are equivalent.  
\end{proof}
\begin{figure}[htb]
    \[ 
    \begin{array}{c} \includegraphics[scale=.4]{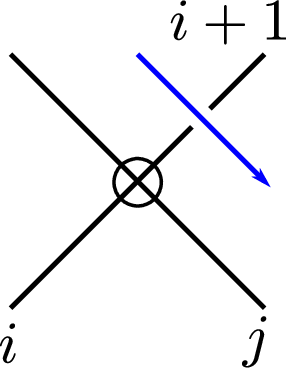}\end{array} \leftrightharpoons_{sw} \begin{array}{c} \includegraphics[scale=.4]{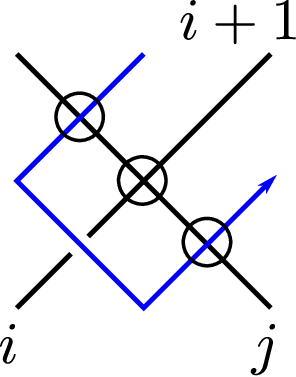}\end{array}  \quad
\begin{array}{c} \includegraphics[scale=.4]{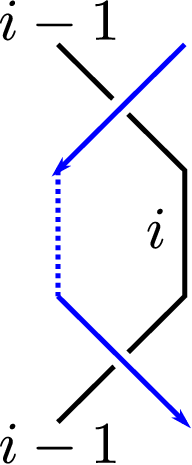}\end{array} \leftrightharpoons_{sw} \begin{array}{c} \includegraphics[scale=.4]{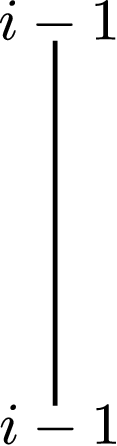}\end{array} \quad 
    \begin{array}{c} \includegraphics[scale=.4]{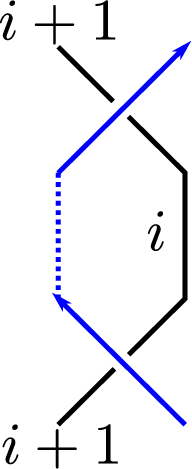}\end{array} \leftrightharpoons_{sw} \begin{array}{c} \includegraphics[scale=.4]{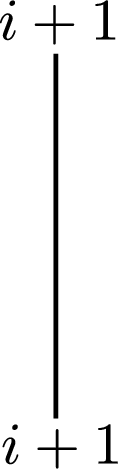}\end{array} 
    \] 
    \caption{The Alexander system 1 move (left), Alexander system 2A move (center), and Alexander system 2B move (right).} \label{fig_alex_system_moves_1_2}
\end{figure}

\begin{figure}[htb]
\[
\begin{array}{cc}
    \begin{array}{c} \includegraphics[scale=.4]{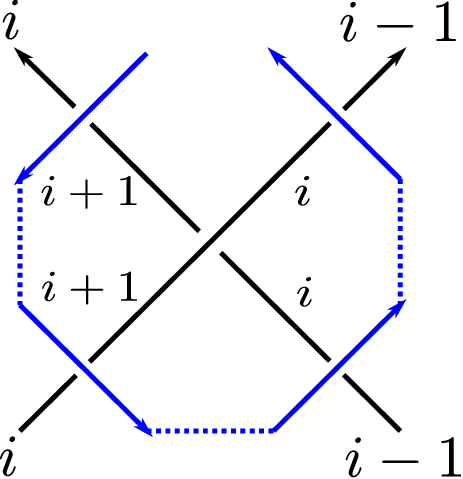}\end{array} \leftrightharpoons_{sw} \begin{array}{c} \includegraphics[scale=.4]{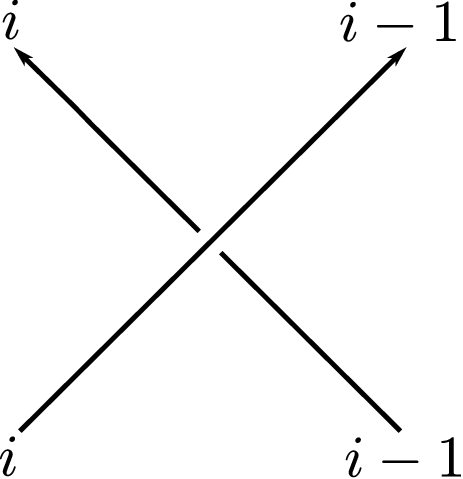}\end{array} 
    & 
    \begin{array}{c} \includegraphics[scale=.4]{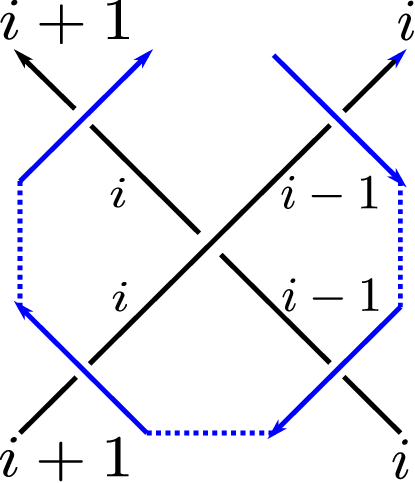}\end{array} \leftrightharpoons_{sw} \begin{array}{c} \includegraphics[scale=.4]{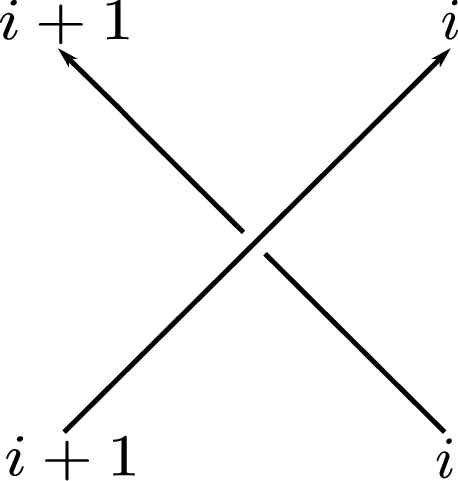}\end{array} 
\end{array}
\]
\caption{The Alexander system 3A (left) and 3B (right) moves.}
\label{fig:ocp3a_z}
\end{figure}

Now consider the moves shown in Figures \ref{fig_alex_system_moves_1_2} and \ref{fig:ocp3a_z}. These will be called \emph{Alexander system (AS) moves}. Each move transforms an Alexander system $(D,\gamma,\Gamma)$ into an equivalent Alexander system $(D,\gamma',\Gamma')$. The Alexander system 1 move on the left in Figure \ref{fig_alex_system_moves_1_2} is a detour move. This has no effect on the Alexander sub-numbering of $D$. The Alexander system 2A and 2B moves (see Figure \ref{fig_alex_system_moves_1_2} center and right) shows the effect of an $\Omega 2$ move on an Alexander sub-numbering of $D$. Note that after performing the move, there are no over-crossings of $\gamma$ and $D$ in the depicted portion of the virtual link diagram. By Lemma \ref{lemma_alex_and_sw}, only the positions of the over-crossing arcs of $\gamma$ and $D$ are of importance. Hence, it is not necessary to draw the $\gamma$ component after the $\Omega 2$ move is performed. Likewise, the loop on the left hand side of each move in Figure \ref{fig:ocp3a_z} can be pulled off using a $\Omega 3$ move and two $\Omega 2$  moves. Afterwards, there is a part of $\gamma$ that has no crossings with $D$. The moves in Figure \ref{fig:ocp3a_z} are then completed by erasing this portion of $\gamma$ and then reconnecting the arcs of $\gamma$ in any desired fashion.  

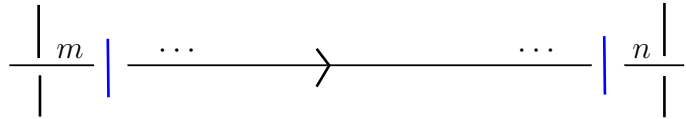
\begin{figure}[htb]
    \begin{center}
\begingroup%
  \makeatletter%
  \providecommand\color[2][]{%
    \errmessage{(Inkscape) Color is used for the text in Inkscape, but the package 'color.sty' is not loaded}%
    \renewcommand\color[2][]{}%
  }%
  \providecommand\transparent[1]{%
    \errmessage{(Inkscape) Transparency is used (non-zero) for the text in Inkscape, but the package 'transparent.sty' is not loaded}%
    \renewcommand\transparent[1]{}%
  }%
  \providecommand\rotatebox[2]{#2}%
  \newcommand*\fsize{\dimexpr\f@size pt\relax}%
  \newcommand*\lineheight[1]{\fontsize{\fsize}{#1\fsize}\selectfont}%
  \ifx\svgwidth\undefined%
    \setlength{\unitlength}{255.26893436bp}%
    \ifx\svgscale\undefined%
      \relax%
    \else%
      \setlength{\unitlength}{\unitlength * \real{\svgscale}}%
    \fi%
  \else%
    \setlength{\unitlength}{\svgwidth}%
  \fi%
  \global\let\svgwidth\undefined%
  \global\let\svgscale\undefined%
  \makeatother%
  \begin{picture}(1,0.16939042)%
    \lineheight{1}%
    \setlength\tabcolsep{0pt}%
    \put(0,0){\includegraphics[width=\unitlength]{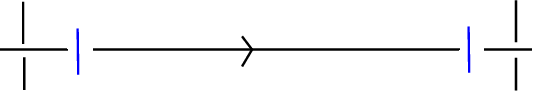}}%
    \put(0.07,0.09){\color[rgb]{0,0,0}\makebox(0,0)[lt]{\lineheight{1.25}\smash{\begin{tabular}[t]{l}$m$\end{tabular}}}}%
    \put(0.92,0.09){\color[rgb]{0,0,0}\makebox(0,0)[lt]{\lineheight{1.25}\smash{\begin{tabular}[t]{l}$n$\end{tabular}}}}%
    \put(0.22,0.09){\color[rgb]{0,0,0}\makebox(0,0)[lt]{\lineheight{1.25}\smash{\begin{tabular}[t]{l}$\cdots$\end{tabular}}}}%
    \put(0.75,0.09){\color[rgb]{0,0,0}\makebox(0,0)[lt]{\lineheight{1.25}\smash{\begin{tabular}[t]{l}$\cdots$\end{tabular}}}}%
  \end{picture}%
\endgroup%

    \caption{The arc $\alpha$ from Lemma \ref{arc_lemma}.}
     \label{arc_lemma_fig}
    \end{center}
\end{figure}

\begin{lemma} \label{arc_lemma} Suppose that $(D,\gamma,\Gamma)$ is an Alexander system. Let $\alpha$ be the short arc of $D$ between two (possibly the same) classical crossings of $D$. Let $m$ and $n$ be the integer labelings of the first and last short arcs along $\alpha$, as shown in Figure \ref{arc_lemma_fig}. Then there is an equivalent Alexander system $(D,\gamma',\Gamma')$ such that $\gamma'$ over-crosses $\alpha$ exactly $|m-n|$ times, where all crossings are positive if $m < n$ and all crossings are negative if $m > n$.
\end{lemma}

\begin{proof} First suppose that all of the over-crossing arcs of $\gamma$ are oriented in the same direction along $\alpha$. Then either the labeling starts at $m$ and decreases by 1 to $n$ while traveling the length of $\alpha$ or it increases by 1 to $n$. This implies that all of the crossing of $\gamma$ with $\alpha$ are negative in the first case and positive in the second. Hence, the number of crossings is $|m-n|$, $m<n$ when the common crossing sign is positive, and $m>n$ when the common crossing sign is negative.
\newline
\newline
For the general case, suppose that not all of the crossings of $\gamma$ and $\alpha$ are oriented in the same direction. Then starting from the left, we can always find a first pair of arcs that have opposite orientations. By Lemma \ref{lemma_alex_and_sw}, there is an equivalent Alexander system such that this pair of arcs occurs consecutively in their ordering along $\gamma$. Using either an Alexander system 2A or 2B move, this pair of crossings can be removed. This process can be repeated until all of the over-crossings of $\gamma$ with $\alpha$ point in the same direction. The result then follows from the first case.  
\end{proof}

\begin{lemma} \label{main_lem} Any two Alexander systems for the same virtual link diagram are equivalent.
\end{lemma}

\begin{proof} Let $(D,\gamma_0,\Gamma_0)$, $(D,\gamma_1,\Gamma_1)$ be two Alexander systems for a virtual link diagram $D$. The Alexander sub-numberings $\Gamma_0,\Gamma_1$ provide instructions on how to achieve the required semi-welded equivalence. Let $x$ be a classical self-crossing of $D$. The labels in $D_{\gamma_i}$ incident to $x$ are determined by a single integer $m_i^x$. If $x$ is rotated so that it appears as in Figure \ref{fig_alex_num_defn} , take $m_i^x$ to be the label in the lower right-hand corner. If $m_0^x=m_1^x$, there is nothing to do. If $m_0^x<m_1^x$, perform Alexander system 3A moves on $(D,\gamma_0,\Gamma_0)$ until $m_0^x=m_1^x$. If $m_0^x>m_1^x$, perform Alexander system 3B moves on $(D,\gamma_0,\Gamma_0)$ until equality is reached. Do this for every classical self-crossing of $D$. Hence we may assume that the two Alexander systems have the same integer labeling at every classical self-crossing of $D$.  Now apply Lemma \ref{arc_lemma} to each short arc $\alpha$ of $D$ for each of the two Alexander systems. After this reduction, the over-crossings of $\gamma_0$ with $D$ must be of the same quantity and sign with the over-crossings of $\gamma_1$ with $D$ along each $\alpha$. Finally, applying Lemma \ref{lemma_alex_and_sw} to $\gamma_0$, we can reorder the over-crossings of $\gamma_0$ with $D$ to match the ordering of the over-crossings of $\gamma_1$ with $D$. It follows that $(D,\gamma_0,\Gamma_0)$, $(D,\gamma_1,\Gamma_1)$ are equivalent Alexander systems for $D$.
\end{proof}

\begin{proof}[Proof of Theorem \ref{thm_main}] Let $(D_{0},\gamma_0,\Gamma_0)$, $(D_{1},\gamma_1,\Gamma_1)$ be Alexander systems with $D_0 \leftrightharpoons D_1$. Lemmas \ref{main_lem} and \ref{sw_eq_thrm} imply that $D_0 \cup \gamma_0 \leftrightharpoons_{sw} \Zh^{\text{op}}(D_0) \leftrightharpoons_{sw} \Zh^{\text{op}}(D_1) \leftrightharpoons_{sw} D_1 \cup \gamma_1$.
\end{proof}

\subsection{Alexander systems $\&$ relative Seifert surfaces} \label{sec_geom_origin} The results of the previous sections can now be used to give a geometric derivation of the $\Zh^{\text{op}}$-construction. By Lemma \ref{lemma_zh_is_alexander}, every virtual link diagram $D$ has an Alexander system $(D,\omega^{\text{op}},\Omega)$ that is Alexander numerable. Since $D \cup \omega^{\text{op}}$ is Alexander numerable, it bounds a Seifert surface $F$ in its Carter surface $\Sigma$ (see e.g. \cite{acpaper}, Theorem 6.1). The boundary of $F$ in $\Sigma \times \mathbb{I}$ consists of a link $\mathscr{L}$ corresponding to the virtual link diagram $D$ and a knot $\mathscr{C}$ corresponding to $\omega^{\text{op}}$. As the only classical crossings of $\omega^{\text{op}}$ and $D$ are over-crossings, $\mathscr{C}$ lies above $\mathscr{L}$ in $\Sigma \times \mathbb{I}$. Furthermore, $\omega^{\text{op}}$ has no classical self-crossings. It follows that $\mathscr{C}$ can be drawn as a simple closed curve in $\Sigma \times \{1\}$. Now, since Alexander systems characterize the $\Zh^{\text{op}}$-construction, our aim is to show that Alexander systems are generated by surfaces $F$ that are cobounded by $\mathscr{L}$ and some simple closed curves in $\Sigma \times \{1\}$. Such surfaces always exist, as is shown below.

\begin{lemma} \label{lemma_cobound_exist} Let $\Sigma$ be the Carter surface of a virtual link diagram $D$ and let $\mathscr{D}$ be the corresponding link diagram on $\Sigma$. Let $\mathscr{L} \subset \Sigma \times \mathbb{I}$ be a link having diagram $\mathscr{D}$. Then there is a compact oriented surface $F \subset \Sigma \times \mathbb{I}$ cobounded by $\mathscr{L}$ and a collection of disjoint simple closed curves $\mathscr{C}_F\subset \Sigma \times \{1\}$.
\end{lemma}
\begin{proof} Note that homology relative to $\Sigma \times \{1\}$ satisfies $H_1(\Sigma \times \mathbb{I}, \Sigma \times \{1\};\mathbb{Z}) \cong \{0\}$. Hence, any link $\mathscr{L} \subset \Sigma \times \mathbb{I}$ is a relative $2$-boundary and the result follows.
\end{proof}

\begin{definition}[Relative Seifert surface] Let $D$ be a virtual link diagram and let $\Sigma$, $\mathscr{D}$, $\mathscr{L}$ be as in Lemma \ref{lemma_cobound_exist}. A \emph{relative Seifert surface for $D$} is a compact oriented surface $F \subset \Sigma \times \mathbb{I}$ with $\partial F=\mathscr{L} \sqcup \mathscr{C}_F$, where $\mathscr{C}_F$ is a collection of disjoint oriented closed curves in $\Sigma \times \{1\}$. 
\end{definition}

Every relative Seifert surface $F$ for $D$ determines an Alexander system $(D,\gamma_F,\Gamma_F)$, as we now show. The link $\mathscr{L} \sqcup \mathscr{C}_F \subset \Sigma \times \mathbb{I}$ projects to a link diagram $\mathscr{D} \cup \mathscr{D}_F$ on $\Sigma \times \{0\}$ (perhaps after  deforming $\mathscr{C}_F$ slightly to ensure transversal intersections). Observe that the only crossings in $\mathscr{D} \cup \mathscr{D}_F$ involving the components of $\mathscr{D}_F$ are over-crossings by $\mathscr{D}_F$. Now, consider the intersection pairing for the subspaces $\mathscr{D}$ and $\mathscr{D}_F$ of $\Sigma$ (see e.g. Dold \cite{dold}, VIII.13.18):
\[
\bullet:H_1(\mathscr{D};\mathbb{Z}) \times H_1(\mathscr{D}_F;\mathbb{Z}) \to H_{1+1-2}(\mathscr{D} \cap \mathscr{D}_F).
\]
Clearly, $H_{0}(\mathscr{D} \cap \mathscr{D}_F;\mathbb{Z}) \cong \mathbb{Z}^{|\mathscr{D} \cap \mathscr{D}_F|}$. In other words, each over-crossing of $\mathscr{D}_F$ with $\mathscr{D}$ contributes a summand of $\mathbb{Z}$. Since $\mathscr{D}_F$ always over-crosses $\mathscr{D}$, the sign of the crossing is determined by the intersection pairing. To see this, write $\mathscr{D}_F$ as a union of its components $\mathscr{D}_1,\mathscr{D}_2,\ldots,\mathscr{D}_m$. Let $\eta_i=[\mathscr{D}_i] \in H_1(\mathscr{D}_F;\mathbb{Z})$  and let $\eta=[\mathscr{D}]\in H_1(\mathscr{D} ;\mathbb{Z})$. Then:
\[
\eta \bullet \left(\sum_{i=1}^m \eta_i \right) \in H_0(\mathscr{D} \cap \mathscr{D}_F;\mathbb{Z})\cong \mathbb{Z}^{|\mathscr{D} \cap \mathscr{D}_F|}
\]
has a $\pm 1$ in each summand according to the local intersect number. See Figure \ref{fig_local_int_num}. It is readily checked that the sign of the crossing agrees with the local intersection number of $\eta$ and $\eta_i$.
\newline

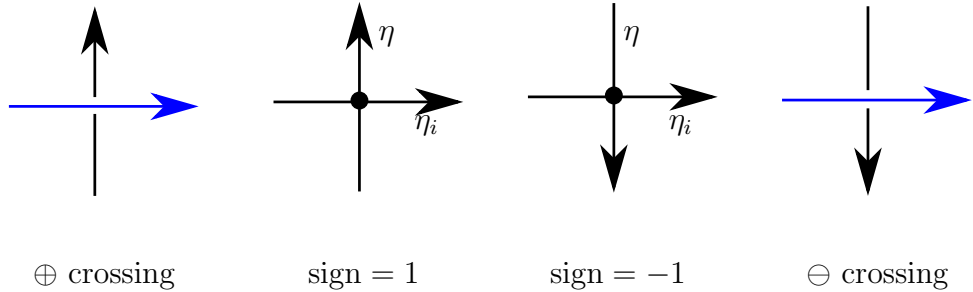
\begin{figure}[htb]
\begin{tabular}{cccc}
\begin{tabular}{c}
\def\svgwidth{1in}
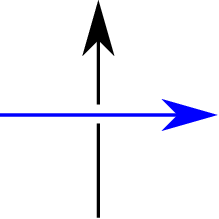 \vspace{.65cm} \end{tabular}  & \begin{tabular}{c}
\def\svgwidth{1in}
\begingroup%
  \makeatletter%
  \providecommand\color[2][]{%
    \errmessage{(Inkscape) Color is used for the text in Inkscape, but the package 'color.sty' is not loaded}%
    \renewcommand\color[2][]{}%
  }%
  \providecommand\transparent[1]{%
    \errmessage{(Inkscape) Transparency is used (non-zero) for the text in Inkscape, but the package 'transparent.sty' is not loaded}%
    \renewcommand\transparent[1]{}%
  }%
  \providecommand\rotatebox[2]{#2}%
  \newcommand*\fsize{\dimexpr\f@size pt\relax}%
  \newcommand*\lineheight[1]{\fontsize{\fsize}{#1\fsize}\selectfont}%
  \ifx\svgwidth\undefined%
    \setlength{\unitlength}{149.65973225bp}%
    \ifx\svgscale\undefined%
      \relax%
    \else%
      \setlength{\unitlength}{\unitlength * \real{\svgscale}}%
    \fi%
  \else%
    \setlength{\unitlength}{\svgwidth}%
  \fi%
  \global\let\svgwidth\undefined%
  \global\let\svgscale\undefined%
  \makeatother%
  \begin{picture}(1,0.69692045)%
    \lineheight{1}%
    \setlength\tabcolsep{0pt}%
    \put(0,0){\includegraphics[width=\unitlength]{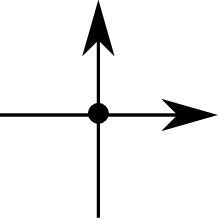}}%
\put(0.73778465,0.33409111){\color[rgb]{0,0,0}\makebox(0,0)[lt]{\lineheight{40.54999924}\smash{\begin{tabular}[t]{l}$\eta_i$\end{tabular}}}}%
    \put(0.55,0.8){\color[rgb]{0,0,0}\makebox(0,0)[lt]{\lineheight{40.54999924}\smash{\begin{tabular}[t]{l}$\eta$\end{tabular}}}}%
  \end{picture}%
\endgroup%
 \end{tabular} &  \begin{tabular}{c}\def\svgwidth{1in}
\begingroup%
  \makeatletter%
  \providecommand\color[2][]{%
    \errmessage{(Inkscape) Color is used for the text in Inkscape, but the package 'color.sty' is not loaded}%
    \renewcommand\color[2][]{}%
  }%
  \providecommand\transparent[1]{%
    \errmessage{(Inkscape) Transparency is used (non-zero) for the text in Inkscape, but the package 'transparent.sty' is not loaded}%
    \renewcommand\transparent[1]{}%
  }%
  \providecommand\rotatebox[2]{#2}%
  \newcommand*\fsize{\dimexpr\f@size pt\relax}%
  \newcommand*\lineheight[1]{\fontsize{\fsize}{#1\fsize}\selectfont}%
  \ifx\svgwidth\undefined%
    \setlength{\unitlength}{149.65973225bp}%
    \ifx\svgscale\undefined%
      \relax%
    \else%
      \setlength{\unitlength}{\unitlength * \real{\svgscale}}%
    \fi%
  \else%
    \setlength{\unitlength}{\svgwidth}%
  \fi%
  \global\let\svgwidth\undefined%
  \global\let\svgscale\undefined%
  \makeatother%
  \begin{picture}(1,0.70038785)%
    \lineheight{1}%
    \setlength\tabcolsep{0pt}%
    \put(0,0){\includegraphics[width=\unitlength]{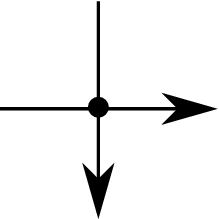}}%
    \put(0.73778465,0.33409111){\color[rgb]{0,0,0}\makebox(0,0)[lt]{\lineheight{40.54999924}\smash{\begin{tabular}[t]{l}$\eta_i$\end{tabular}}}}%
    \put(0.5,0.8){\color[rgb]{0,0,0}\makebox(0,0)[lt]{\lineheight{40.54999924}\smash{\begin{tabular}[t]{l}$\eta$\end{tabular}}}}%
  \end{picture}%
\endgroup%
 \end{tabular} & \begin{tabular}{c}\def\svgwidth{1in}
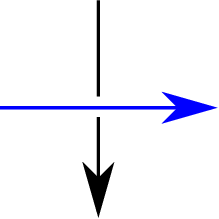 \vspace{.65cm} \end{tabular}  \\
 $\oplus$ crossing & $\text{sign}=1$ & $\text{sign}=-1$ & $\ominus$ crossing
\end{tabular}
\caption{Classical crossings and local intersection numbers.} \label{fig_local_int_num}
\end{figure}

An Alexander system $(D,\gamma_F,\Gamma_F)$ is now constructed as follows. For the $(n+1)$ component virtual link diagram $D \cup \gamma_F$, note that every transversal intersection of $\mathscr{D}$ and $\mathscr{D}_F$ corresponds to a point $x$ on the virtual link diagram $D$. Draw a small over-crossing arc at $x$ so that the sign of the crossing matches the local intersection number of $\mathscr{D}$ and $\mathscr{D}_F$. As in the $\Zh^{\text{op}}$-construction, these are connected in an arbitrary fashion to obtain a single component $\gamma_F$. For the labeling function $\Gamma_F:D_{\gamma_F} \to \mathbb{Z}$, we use the fact that $\mathscr{L} \cup \mathscr{C}_F$ is almost classical. Since $\mathscr{L} \sqcup \mathscr{C}_F$ bounds the Seifert surface $F$, the diagram $\mathscr{D} \cup \mathscr{D}_F$ is Alexander numerable. As in the proof of Lemma \ref{lemma_alex_num_is_system}, the restriction of this labeling to the short arcs of $D_{\gamma_F}$ is an Alexander sub-numbering. This proves the following theorem.

\begin{theorem} \label{thm_rss} Any relative Seifert surface $F$ of $D$ gives is an Alexander system $(D,\gamma_F,\Gamma_F)$. 
\end{theorem}

An additional feature of the above construction is that it explains the need for the semi-welded equivalence relation. Each over-crossing arc of $\gamma_F$ with $D$ represents a summand of the intersection homology class of $\eta \bullet \sum_i \eta_i \in H_0(\mathscr{D} \cap \mathscr{D}_F;\mathbb{Z})$. Since $H_0(\mathscr{D} \cap \mathscr{D}_F;\mathbb{Z})$ is abelian, the order in which these summands are added together to make the full intersection homology class $\eta \bullet \sum_i \eta_i$ is irrelevant. The $\gamma_F$ component in $D \cup \gamma_F$ can therefore be interpreted as a diagrammatic representation of the full intersection homology class $\eta \bullet \sum_i \eta_i$.  The $\omega$OCC relation simply encodes the fact that $H_0(\mathscr{D} \cap \mathscr{D}_F;\mathbb{Z})$ is abelian. Next we derive the $\Zh^{\text{op}}$-construction from relative Seifert surfaces.

\begin{theorem} For every virtual link diagram $D$, there is a relative Seifert surface $S$ such that the Alexander system $(D,\gamma_S,\Gamma_S)$ is equivalent to the Alexander system $(D,\omega^{\text{op}},\Omega)$ obtained from the $\Zh^{\text{op}}$-construction.
\end{theorem}

\begin{proof} Let $\Sigma$ be the Carter surface of $D$ and $\mathscr{D}$ the diagram of $D$ on $\Sigma$. Perform the oriented smoothing at each crossing of $\mathscr{D}$ on $\Sigma$ (see Figure \ref{zh_surface}). As usual, we would like to connect the smoothed crossings together to make the Seifert circles of $\mathscr{D}$. To do so, however, it may be necessary to introduce new intersections with the arcs of $\mathscr{D}$. Therefore, we must first redraw the smoothed crossings of $\mathscr{D}$ in such a way that no additional crossings are needed. This is shown in the two middle pictures in Figure \ref{zh_surface}. Note that as you enter or exit the crossing of $\mathscr{D}$, the arcs of the smoothed crossing now always lie to the right of an arc of $\mathscr{D}$. Hence the smoothed arcs for adjacent crossings of $\mathscr{D}$ can be connected together without introducing any additional crossings with $\mathscr{D}$. At each crossing, the figure shows a piece of the desired two-sided surface $S$. Extending these pieces of the surface along the remaining arcs of $\mathscr{D}$ gives an oriented surface $S$ with boundary $\mathscr{L} \sqcup \mathscr{C}_S$.  The fact that $D \cup \gamma_S$ coincides with  $\Zh^{\text{op}}(D)$, up to semi-welded equivalence, follows from Figure \ref{zh_surface}, right.
\end{proof}

\begin{figure}[htb]
\[
\xymatrix{
\begin{array}{c} \def\svgwidth{1.15in}
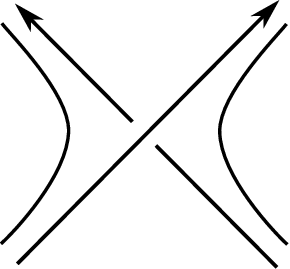 \end{array} \ar[r] & \begin{array}{c} \def\svgwidth{1in}
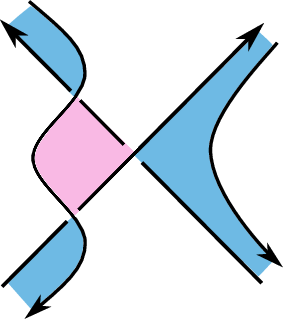 \end{array} \ar[r] & \begin{array}{c} \def\svgwidth{1in}
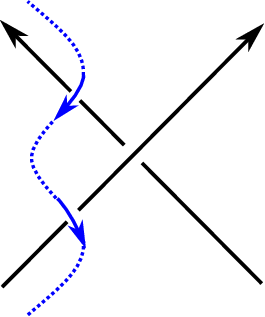 \end{array}\\
\begin{array}{c} \def\svgwidth{1.15in}
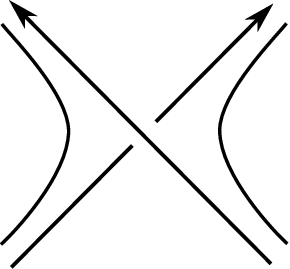 \end{array} \ar[r] & \begin{array}{c} \def\svgwidth{1in}
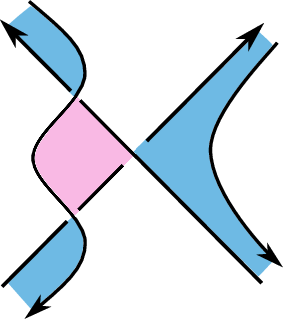 \end{array} \ar[r] & \begin{array}{c} \def\svgwidth{1in}
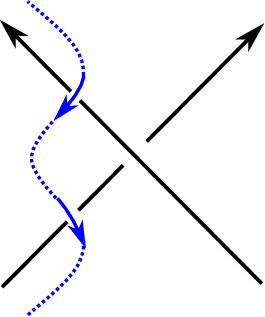 \end{array} 
}
\]
\caption{Realizing $\Zh^{\text{op}}(D)$ with a relative Seifert surface.}
\label{zh_surface}
\end{figure}

As an application of the above, we give a geometric proof of Theorem \ref{thm_splits}. A corollary of this is that $\Zh^{\text{op}}(D)$ splits as an Alexander system $D \sqcup \bigcirc$ if and only if $D$ is almost classical.

\begin{theorem} \label{thm_zh_splits} Let $D$ be a diagram of an almost classical link. Then $\Zh^{\text{op}}(D)$ is semi-welded equivalent to the split link $D \sqcup \omega^{\text{op}}$, where $\omega^{\text{op}}$ is an unknot.
\end{theorem}
\begin{proof} If $D_0,D_1$ are equivalent diagrams, then $\Zh^{\text{op}}(D_0),\Zh^{\text{op}}(D_1)$ are semi-welded equivalent. Hence we may assume that the given diagram $D$ is Alexander numerable. Let $\Sigma$ be the Carter surface of $D$ and $\mathscr{D}$ the corresponding diagram of $D$ on $\Sigma$. Any link $\mathscr{L} \subset \Sigma \times \mathbb{I}$ with diagram $\mathscr{D}$ bounds a Seifert surface $F$ in $\Sigma \times \mathbb{I}$. Let $\tau:[0,1] \to \Sigma \times \mathbb{I}$ be a smooth path such that $\tau(0)$ is in the interior of $F$, $\tau(1) \in \Sigma \times \{1\}$, and $\tau((0,1)) \cap F=\varnothing$. Let $N$ be a sufficiently small closed tubular neighborhood of the image of $\tau$ so that $N$ intersects $F$ exactly in a disc $E_0 \subset \text{int}(F)$ that is centered at $\tau(0)$ and $N$ intersects $\Sigma \times \{1\}$ only in a disc $E_1$ that is centered at $\tau(1)$ . Let $T=\partial N \smallsetminus (\text{int}(E_0) \cup \text{int}(E_1))$, so that $T$ is an annulus. Define $S=(F \smallsetminus \text{int}(E_0))  \cup T$. Then $S$ intersects $\Sigma \times \{1\}$ in the circle $\partial E_1$ that bounds the disc $E_1$ in $\Sigma \times \{1\}$. Let $(D,\gamma_S,\Gamma_S)$ be the Alexander system for $D$ corresponding to $S$ as in Theorem \ref{thm_rss}. Since $\partial E_1$ can be contracted to as small a circle in $\Sigma$ as desired, it follows that $\gamma_S$ has no classical or virtual crossings with $D$. Hence, $D \sqcup \gamma_S$ splits as desired. 
\end{proof}

\begin{corollary} A virtual link diagram $D$ is almost classical if and only if $\Zh^{\text{op}}(D)$ is semi-welded equivalent to a split diagram $D' \sqcup \bigcirc$ where $(D',\bigcirc,\Gamma')$ is an Alexander system.
\end{corollary}
\begin{proof} If $D$ is almost classical, then $\Zh^{\text{op}}(D)$ splits by Theorem \ref{thm_zh_splits}. Conversely, suppose $\Zh^{\text{op}}(D) \leftrightharpoons_{sw} D' \sqcup \bigcirc$. Since $(D',\bigcirc,\Gamma')$ is an Alexander system, Corollary \ref{cor_split_iff_alex_num}, implies that $D'$ is an Alexander numerable diagram. Since the $\omega$OCC move can only be applied when the over-crossing arc is in the $\omega^{\text{op}}$ component, it follows that $D \leftrightharpoons D'$ as virtual link diagrams. Hence, $D$ is almost classical. 
\end{proof}

\section{Realizing the Dye-Kauffman-Miyazawa polynomial with \texorpdfstring{$\Zh$}{Lg}}\label{AMpoly}

\subsection{Review of DKM-polynomial}  Here we will use the formulation due to Dye-Kauffman \cite{dye2009virtual}, which gives an equivalent invariant to the Miyazawa polynomial \cite{miyazawa2008}. The DKM-polynomial is defined as a state sum, similar to the Kauffman bracket, but over the ring $\mathbb{Z}[A^{\pm1}, K_{1},K_{2}, \ldots]$. By an \emph{arrow state} of $D$, we mean the collection of planar curves obtained from $D$ by successively applying the following rule to the classical crossings of $D$ until no classical crossings remain:

\begin{align*} 
   \left \langle \left \langle 
   \begin{array}{c}\includegraphics[height=.045\textheight]{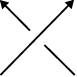}  \end{array} \right \rangle \right \rangle & =  A \left \langle \left \langle  \begin{array}{c}\includegraphics[height=.045\textheight]{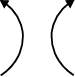} \end{array} \right\rangle \right\rangle + A^{-1} \left \langle \left \langle \begin{array}{c}\includegraphics[height=.045\textheight]{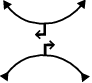} \end{array}  \right \rangle \right \rangle \label{asr_pos}, \\
   \left \langle \left\langle \begin{array}{c} \includegraphics[height=.045\textheight]{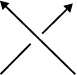} \end{array}  \right \rangle \right\rangle &=  A^{-1} \left \langle \left\langle \begin{array}{c}\includegraphics[height=.045\textheight]{smooth_oriented.eps} \end{array} \right \rangle \right\rangle + A \left \langle \left\langle\begin{array}{c} \includegraphics[height=.045\textheight]{smooth_disoriented.eps} \end{array} \right \rangle \right \rangle 
\end{align*}

Note that in addition to the usual state curves of a virtual knot diagram, an arrow state involving at least one $B$ smoothing contains some added symbols $\Rsh$. Furthermore, the curves between the endpoints of these symbols are oriented in the manner shown above. The added symbols $\Rsh$ are called \emph{poles}. The evaluation $\langle S \rangle$ for a arrow state $S$ is now determined as follows. First, we reduce the poles $\Rsh$ along all the state curves in $S$. If two neighboring poles are on the same side of the component on which they sit (or equivalently, they point in the same direction), cancel the two poles and adjust the alternating orientation. See Figure \ref{arrow_reduction}. This rule is continually applied until no additional reductions are possible. For a loop $C$ in a state $S$, set $a(C) \in \mathbb{N} \cup \{0\}$ to be half of the number of poles on $C$ and set $K_0=1$. Define the evaluation $[S ]$ of the arrow state $S$ to be:
\[
[S]=\prod_{C \in S}K_{a(C)}
\]
\begin{definition}[DKM-polynomial] Let $D$ be an oriented virtual link diagram. For an arrow state $S$ of $D$, let $\alpha(S)$ be the number of $A$-smoothings, $\beta(S)$ the number of $B$-smoothings, $|S|$ is the number of loops in the state. Let $w(D)$ denote the writhe of $D$. Then the \emph{DKM-polynomial} $V_{DKM}(D)$ is the virtual isotopy invariant of $D$ determined by the following normalized state sum:
\[
 \langle\langle D \rangle\rangle = \sum_{S} A^{\alpha(S) - \beta(S)} (-A^2-A^{-2})^{|S|-1}[ S ], \quad V_{DKM}(D) = (-A)^{-3\cdot w(D)} \langle\langle D \rangle\rangle
\]
For a proof that this defines a virtual isotopy invariant, the reader is referred to \cite{dye2009virtual}. Example calculations and a tabulation can be found in Bhandari-Dye-Kauffman \cite{bdk}.
\end{definition}

\begin{figure}[htb]
\def\svgwidth{3in}
\begingroup%
  \makeatletter%
  \providecommand\color[2][]{%
    \errmessage{(Inkscape) Color is used for the text in Inkscape, but the package 'color.sty' is not loaded}%
    \renewcommand\color[2][]{}%
  }%
  \providecommand\transparent[1]{%
    \errmessage{(Inkscape) Transparency is used (non-zero) for the text in Inkscape, but the package 'transparent.sty' is not loaded}%
    \renewcommand\transparent[1]{}%
  }%
  \providecommand\rotatebox[2]{#2}%
  \newcommand*\fsize{\dimexpr\f@size pt\relax}%
  \newcommand*\lineheight[1]{\fontsize{\fsize}{#1\fsize}\selectfont}%
  \ifx\svgwidth\undefined%
    \setlength{\unitlength}{151.50602806bp}%
    \ifx\svgscale\undefined%
      \relax%
    \else%
      \setlength{\unitlength}{\unitlength * \real{\svgscale}}%
    \fi%
  \else%
    \setlength{\unitlength}{\svgwidth}%
  \fi%
  \global\let\svgwidth\undefined%
  \global\let\svgscale\undefined%
  \makeatother%
  \begin{picture}(1,0.1031069)%
    \lineheight{1}%
    \setlength\tabcolsep{0pt}%
    \put(0,0){\includegraphics[width=\unitlength]{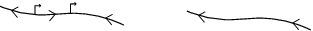}}%
    \put(0.45656702,0.06153754){\color[rgb]{0,0,0}\makebox(0,0)[lt]{\lineheight{1.25}\smash{\begin{tabular}[t]{l}$\rightleftharpoons$\end{tabular}}}}%
  \end{picture}%
\endgroup%
 
\normalsize
\caption{The rule for canceling poles.}
\label{arrow_reduction}
\end{figure}

\begin{example} Let $D$ be the right-handed virtual trefoil shown on the left in Figure \ref{vtref_states}. The right side in Figure \ref{vtref_states} shows the four arrow states.  Then $\langle \langle D \rangle \rangle, V_{DKM}(D)$ are given by:
\[
\langle \langle D \rangle\rangle=A^{2}+(1-A^{-4})K_{1}, \quad  V_{DKM}(D)=(-A)^{-6} (A^{2}+(1-A^{-4})K_{1}).
\]
\end{example}

 \begin{figure}[htb]
    \centering
    \begin{tabular}{c|cccc}
     \includegraphics[scale=.7]{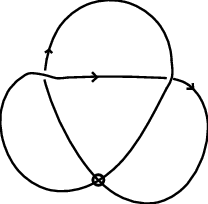}&
     \def\svgwidth{1.1in} 
\begingroup%
  \makeatletter%
  \providecommand\color[2][]{%
    \errmessage{(Inkscape) Color is used for the text in Inkscape, but the package 'color.sty' is not loaded}%
    \renewcommand\color[2][]{}%
  }%
  \providecommand\transparent[1]{%
    \errmessage{(Inkscape) Transparency is used (non-zero) for the text in Inkscape, but the package 'transparent.sty' is not loaded}%
    \renewcommand\transparent[1]{}%
  }%
  \providecommand\rotatebox[2]{#2}%
  \newcommand*\fsize{\dimexpr\f@size pt\relax}%
  \newcommand*\lineheight[1]{\fontsize{\fsize}{#1\fsize}\selectfont}%
  \ifx\svgwidth\undefined%
    \setlength{\unitlength}{100.06739044bp}%
    \ifx\svgscale\undefined%
      \relax%
    \else%
      \setlength{\unitlength}{\unitlength * \real{\svgscale}}%
    \fi%
  \else%
    \setlength{\unitlength}{\svgwidth}%
  \fi%
  \global\let\svgwidth\undefined%
  \global\let\svgscale\undefined%
  \makeatother%
  \begin{picture}(1,1.11155544)%
    \lineheight{1}%
    \setlength\tabcolsep{0pt}%
    \put(0,0){\includegraphics[width=\unitlength]{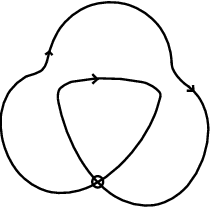}}%
    \put(0.32,-0.1){\color[rgb]{0,0,0}\makebox(0,0)[lt]{\lineheight{1.25}\smash{\begin{tabular}[t]{l}$A^{2}$\end{tabular}}}}%
  \end{picture}%
\endgroup%
   &   \def\svgwidth{1.1in} 
\begingroup%
  \makeatletter%
  \providecommand\color[2][]{%
    \errmessage{(Inkscape) Color is used for the text in Inkscape, but the package 'color.sty' is not loaded}%
    \renewcommand\color[2][]{}%
  }%
  \providecommand\transparent[1]{%
    \errmessage{(Inkscape) Transparency is used (non-zero) for the text in Inkscape, but the package 'transparent.sty' is not loaded}%
    \renewcommand\transparent[1]{}%
  }%
  \providecommand\rotatebox[2]{#2}%
  \newcommand*\fsize{\dimexpr\f@size pt\relax}%
  \newcommand*\lineheight[1]{\fontsize{\fsize}{#1\fsize}\selectfont}%
  \ifx\svgwidth\undefined%
    \setlength{\unitlength}{100.06739044bp}%
    \ifx\svgscale\undefined%
      \relax%
    \else%
      \setlength{\unitlength}{\unitlength * \real{\svgscale}}%
    \fi%
  \else%
    \setlength{\unitlength}{\svgwidth}%
  \fi%
  \global\let\svgwidth\undefined%
  \global\let\svgscale\undefined%
  \makeatother%
  \begin{picture}(1,1.11155544)%
    \lineheight{1}%
    \setlength\tabcolsep{0pt}%
    \put(0,0){\includegraphics[width=\unitlength]{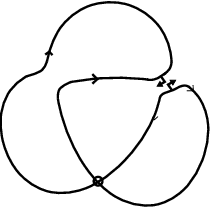}}%
    \put(0.32,-0.10){\color[rgb]{0,0,0}\makebox(0,0)[lt]{\lineheight{1.25}\smash{\begin{tabular}[t]{l}$K_{1}$\end{tabular}}}}%
  \end{picture}%
\endgroup%
 
     & \vspace{20pt}
        \def\svgwidth{1.1in}   
\begingroup%
  \makeatletter%
  \providecommand\color[2][]{%
    \errmessage{(Inkscape) Color is used for the text in Inkscape, but the package 'color.sty' is not loaded}%
    \renewcommand\color[2][]{}%
  }%
  \providecommand\transparent[1]{%
    \errmessage{(Inkscape) Transparency is used (non-zero) for the text in Inkscape, but the package 'transparent.sty' is not loaded}%
    \renewcommand\transparent[1]{}%
  }%
  \providecommand\rotatebox[2]{#2}%
  \newcommand*\fsize{\dimexpr\f@size pt\relax}%
  \newcommand*\lineheight[1]{\fontsize{\fsize}{#1\fsize}\selectfont}%
  \ifx\svgwidth\undefined%
    \setlength{\unitlength}{100.06739044bp}%
    \ifx\svgscale\undefined%
      \relax%
    \else%
      \setlength{\unitlength}{\unitlength * \real{\svgscale}}%
    \fi%
  \else%
    \setlength{\unitlength}{\svgwidth}%
  \fi%
  \global\let\svgwidth\undefined%
  \global\let\svgscale\undefined%
  \makeatother%
  \begin{picture}(1,1.11155544)%
    \lineheight{1}%
    \setlength\tabcolsep{0pt}%
    \put(0,0){\includegraphics[width=\unitlength]{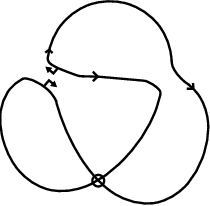}}%
    \put(0.32,-0.1){\color[rgb]{0,0,0}\makebox(0,0)[lt]{\lineheight{1.25}\smash{\begin{tabular}[t]{l}$K_{1}$\end{tabular}}}}%
  \end{picture}%
\endgroup%
&   \def\svgwidth{1.1in} 
\begingroup%
  \makeatletter%
  \providecommand\color[2][]{%
    \errmessage{(Inkscape) Color is used for the text in Inkscape, but the package 'color.sty' is not loaded}%
    \renewcommand\color[2][]{}%
  }%
  \providecommand\transparent[1]{%
    \errmessage{(Inkscape) Transparency is used (non-zero) for the text in Inkscape, but the package 'transparent.sty' is not loaded}%
    \renewcommand\transparent[1]{}%
  }%
  \providecommand\rotatebox[2]{#2}%
  \newcommand*\fsize{\dimexpr\f@size pt\relax}%
  \newcommand*\lineheight[1]{\fontsize{\fsize}{#1\fsize}\selectfont}%
  \ifx\svgwidth\undefined%
    \setlength{\unitlength}{123.87756157bp}%
    \ifx\svgscale\undefined%
      \relax%
    \else%
      \setlength{\unitlength}{\unitlength * \real{\svgscale}}%
    \fi%
  \else%
    \setlength{\unitlength}{\svgwidth}%
  \fi%
  \global\let\svgwidth\undefined%
  \global\let\svgscale\undefined%
  \makeatother%
  \begin{picture}(1,0.87712078)%
    \lineheight{1}%
    \setlength\tabcolsep{0pt}%
    \put(0,0){\includegraphics[width=\unitlength]{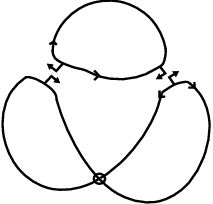}}%
    \put(0.32,-0.1){\color[rgb]{0,0,0}\makebox(0,0)[lt]{\lineheight{1.25}\smash{\begin{tabular}[t]{l}$A^{-2} \cdot d \cdot K_{1}$\end{tabular}}}}%
  \end{picture}%
\endgroup%
 
    \end{tabular}
    
    \caption{The arrow states for the right handed virtual trefoil.}
    \label{vtref_states}
\end{figure}
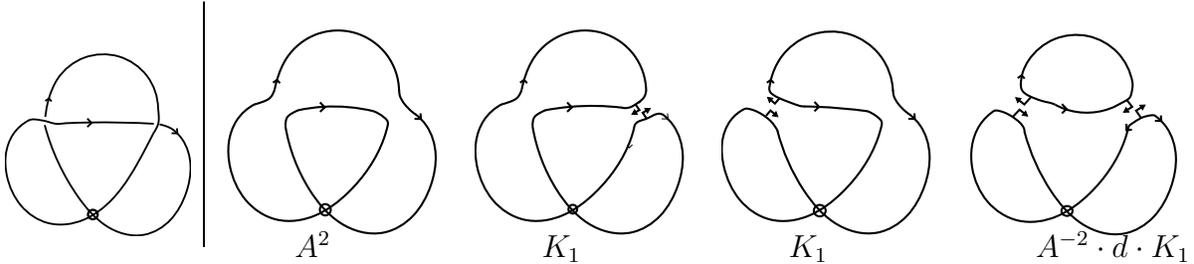

\subsection{The $\Zh$-polynomial} 
Next we define the $\Zh$-bracket, whose normalized $\Zh$-polynomial will be shown to be equal to the DKM-polynomial in the following section. Unlike the DKM-polynomial, the $\Zh$-bracket is defined using the standard state expansion. Let $D$ be a virtual link diagram and begin by constructing $\Zh^{\text{op}}(D)$. A \emph{$\Zh$-state} of $D$ is obtained by choosing the $A$- or $B$-smoothing at a classical crossing of $D$ and expanding according to the Kauffman bracket skein relation: 
\[
\left< \left< \begin{array}{c} \includegraphics[height=.045\textheight]{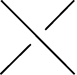} \end{array}\right>\right>_{\zh}= A \left< \left< \begin{array}{c} \includegraphics[height=.045\textheight]{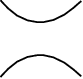} \end{array}\right> \right>_{\zh}+A^{-1} \left<\left<\begin{array}{c} \includegraphics[height=.045\textheight]{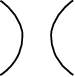} \end{array} \right>\right>_{\zh}.
\]
In the construction of the $\Zh$-state, observe that all crossings involving the $\omega^{\text{op}}$-component are ignored. Hence, a $\Zh$-state consists of the usual state curves of $D$ for the same choice of $A$- or $B$-smoothing at each crossing of $D$, together with the $\omega^{\text{op}}$-component. The $\omega^{\text{op}}$-component only has over-crossings with the state curves of $D$. The $\omega^{\text{op}}$-component retains its orientation, but note that the state curves of $D$ are unoriented. Two $\Zh$-states are considered the same if they may be obtained from one another by semi-welded equivalence.
\newline
\newline
For each $\Zh$-state $S$, define a term $[S]_{\zh}$ in the variables $Z_1,Z_2,\ldots$ as follows. First, give each state curve $C$ in $S$, $C \ne \omega^{\text{op}}$, an arbitrary orientation. This will be denoted $\vec{C}$. Let $\vlk(J,J')$ denote the virtual linking number which counts, with sign, the number of times $J$ over-crosses $J'$. Assign to each state curve $C$ the variable $Z_{|\vlk(\omega^{\text{op}},\vec{C})|/2}$. Note that $\vlk(\omega,\vec{C})$ must indeed be even because $\omega^{\text{op}}$ has two classical crossings with $D$ near each classical crossing of $D$. If either the $A$- or $B$-smoothing is performed at a crossing of $D$, the number of times $\omega^{\text{op}}$ crosses each component remains even. See Figure \ref{fig:zh_poles}. In the case that $\vlk(\omega^{\text{op}},\vec{C})=0$, set $Z_0=1$. Then define for each $\Zh$-state $S$ of $D$ a term $[S]_{\zh}$ as follows:
\[
[S]_{\zh}=\prod_{C \in S}    Z_{|\vlk(\omega^{\text{op}},\vec{C})|/2}
\]
\begin{definition}[$\Zh$-bracket $\&$ $\Zh$-polynomial] Let $D$ be an oriented virtual link diagram. For a $\Zh$-state $S$ of $D$, let $\alpha(S)$ be the number of $A$-smoothings, $\beta(S)$ the number of $B$-smoothings, $|S|$ is the number of loops in the state, excluding the $\omega^{\text{op}}$-component. Let $w(D)$ denote the writhe of $D$. The $\Zh$-\emph{bracket} $\langle \langle D \rangle \rangle_{\zh}$ and $\Zh$-\emph{polynomial} $V_{\zh}(D)$ are defined by the following expressions:
\[
 \langle\langle D \rangle\rangle_{\zh} = \sum_{\zh\text{-states } S} A^{\alpha(S) - \beta(S)} (-A^2-A^{-2})^{|S|-1}[ S ]_{\zh}, \quad V_{\zh}(D) = (-A)^{-3\cdot w(D)} \langle\langle D \rangle\rangle_{\zh}.
\]
\end{definition}

\begin{proposition} The $\Zh$-polynomial is a well-defined invariant of oriented virtual links.
\end{proposition}
\begin{proof} The virtual linking number $\vlk(J,J')$ is an invariant of welded links (see e.g. \cite{c_ext}, Proposition 4.1.2). Since the subscripts of the variables $Z_i$ are absolute values of virtual linking numbers, the value of the invariant is independent of the chosen orientation of the $\Zh$-state curves. The value of the invariant is also independent of order of the over-crossing arcs of the $\omega^{\text{op}}$-component. This follows from the fact that $\langle\langle D \rangle\rangle$ is invariant under the semi-welded move. Indeed, the classical crossings of $\omega^{\text{op}}$ are not smoothed in the state expansion. Hence, the $\Zh$-states on the left-hand side of the move are semi-welded equivalent to the corresponding $\Zh$-states on the right-hand side of the move. The claim then follows since the virtual linking number is a welded link invariant.
\newline
\newline
Lastly, it must be shown that the $\Zh$-polynomial is invariant under virtual Reidemeister equivalence. Clearly, the $\Zh$-polynomial is unaffected by detour moves. For the classical Reidemeister moves, observe that the $\omega^{\text{op}}$-component can always be pulled off of both sides of the move using semi-welded equivalence (see \cite{bbc}, Figures 13, 14, and 15). Invariance under moves $\Omega 1$, $\Omega 2$, and $\Omega 3$ then follows as in the classical case of the Kauffman bracket. \end{proof}

\begin{example} \label{example_v_tref_zh_poly} Let $D$ again be the right-handed virtual trefoil. Figure \ref{fig:tref_zh_states} shows the four $\Zh$-states of $D$ and their corresponding weights. For the leftmost $\Zh$-state, we have $[S]_{\zh}=1$. For the last three $\Zh$-states, $[S]_{\zh}=Z_1$. Then we have:
\begin{align*}
\langle \langle D \rangle \rangle_{\zh} &= A^{2}\cdot 1 +(2+A^{-2}\cdot (-A^2-A^{-2})) Z_1 \\
&= A^{2}+(1-A^{-4}) Z_1, \,\text{and} \\ V_{\zh}(D) &=(-A)^{-6}(A^{2}+(1-A^{-4}) Z_1). 
\end{align*}
\end{example}

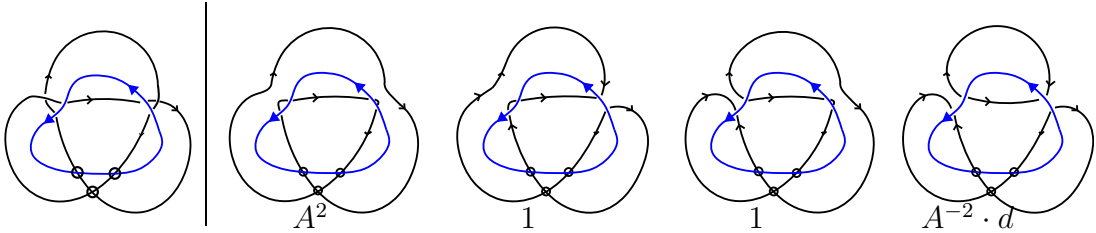
\begin{figure}[htb]
    \centering
    \begin{tabular}{c| c c c c}
     \includegraphics[scale=.7]{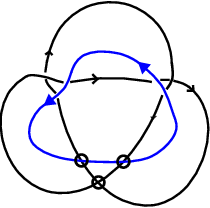} & \def\svgwidth{1.0in} 
\begingroup%
  \makeatletter%
  \providecommand\color[2][]{%
    \errmessage{(Inkscape) Color is used for the text in Inkscape, but the package 'color.sty' is not loaded}%
    \renewcommand\color[2][]{}%
  }%
  \providecommand\transparent[1]{%
    \errmessage{(Inkscape) Transparency is used (non-zero) for the text in Inkscape, but the package 'transparent.sty' is not loaded}%
    \renewcommand\transparent[1]{}%
  }%
  \providecommand\rotatebox[2]{#2}%
  \newcommand*\fsize{\dimexpr\f@size pt\relax}%
  \newcommand*\lineheight[1]{\fontsize{\fsize}{#1\fsize}\selectfont}%
  \ifx\svgwidth\undefined%
    \setlength{\unitlength}{100.06739044bp}%
    \ifx\svgscale\undefined%
      \relax%
    \else%
      \setlength{\unitlength}{\unitlength * \real{\svgscale}}%
    \fi%
  \else%
    \setlength{\unitlength}{\svgwidth}%
  \fi%
  \global\let\svgwidth\undefined%
  \global\let\svgscale\undefined%
  \makeatother%
  \begin{picture}(1,1.11155544)%
    \lineheight{1}%
    \setlength\tabcolsep{0pt}%
    \put(0,0){\includegraphics[width=\unitlength]{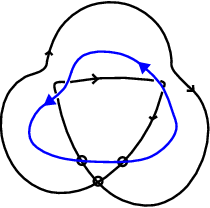}}%
    \put(0.33219464,-0.07){\color[rgb]{0,0,0}\makebox(0,0)[lt]{\lineheight{1.25}\smash{\begin{tabular}[t]{l}$A^{2}$\end{tabular}}}}%
  \end{picture}%
\endgroup%
         & \def\svgwidth{1.0in} 
\begingroup%
  \makeatletter%
  \providecommand\color[2][]{%
    \errmessage{(Inkscape) Color is used for the text in Inkscape, but the package 'color.sty' is not loaded}%
    \renewcommand\color[2][]{}%
  }%
  \providecommand\transparent[1]{%
    \errmessage{(Inkscape) Transparency is used (non-zero) for the text in Inkscape, but the package 'transparent.sty' is not loaded}%
    \renewcommand\transparent[1]{}%
  }%
  \providecommand\rotatebox[2]{#2}%
  \newcommand*\fsize{\dimexpr\f@size pt\relax}%
  \newcommand*\lineheight[1]{\fontsize{\fsize}{#1\fsize}\selectfont}%
  \ifx\svgwidth\undefined%
    \setlength{\unitlength}{100.06739044bp}%
    \ifx\svgscale\undefined%
      \relax%
    \else%
      \setlength{\unitlength}{\unitlength * \real{\svgscale}}%
    \fi%
  \else%
    \setlength{\unitlength}{\svgwidth}%
  \fi%
  \global\let\svgwidth\undefined%
  \global\let\svgscale\undefined%
  \makeatother%
  \begin{picture}(1,1.11155544)%
    \lineheight{1}%
    \setlength\tabcolsep{0pt}%
    \put(0,0){\includegraphics[width=\unitlength]{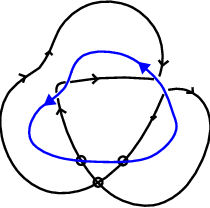}}%
    \put(0.33219464,-0.07){\color[rgb]{0,0,0}\makebox(0,0)[lt]{\lineheight{1.25}\smash{\begin{tabular}[t]{l}$1$\end{tabular}}}}%
  \end{picture}%
\endgroup%
 
      &   \def\svgwidth{1.0in} 
\begingroup%
  \makeatletter%
  \providecommand\color[2][]{%
    \errmessage{(Inkscape) Color is used for the text in Inkscape, but the package 'color.sty' is not loaded}%
    \renewcommand\color[2][]{}%
  }%
  \providecommand\transparent[1]{%
    \errmessage{(Inkscape) Transparency is used (non-zero) for the text in Inkscape, but the package 'transparent.sty' is not loaded}%
    \renewcommand\transparent[1]{}%
  }%
  \providecommand\rotatebox[2]{#2}%
  \newcommand*\fsize{\dimexpr\f@size pt\relax}%
  \newcommand*\lineheight[1]{\fontsize{\fsize}{#1\fsize}\selectfont}%
  \ifx\svgwidth\undefined%
    \setlength{\unitlength}{100.06739044bp}%
    \ifx\svgscale\undefined%
      \relax%
    \else%
      \setlength{\unitlength}{\unitlength * \real{\svgscale}}%
    \fi%
  \else%
    \setlength{\unitlength}{\svgwidth}%
  \fi%
  \global\let\svgwidth\undefined%
  \global\let\svgscale\undefined%
  \makeatother%
  \begin{picture}(1,1.11155544)%
    \lineheight{1}%
    \setlength\tabcolsep{0pt}%
    \put(0,0){\includegraphics[width=\unitlength]{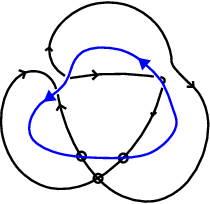}}%
    \put(0.33219464,-0.07){\color[rgb]{0,0,0}\makebox(0,0)[lt]{\lineheight{1.25}\smash{\begin{tabular}[t]{l}$1$\end{tabular}}}}%
  \end{picture}%
\endgroup%
  &  \def\svgwidth{1.0in}
\begingroup%
  \makeatletter%
  \providecommand\color[2][]{%
    \errmessage{(Inkscape) Color is used for the text in Inkscape, but the package 'color.sty' is not loaded}%
    \renewcommand\color[2][]{}%
  }%
  \providecommand\transparent[1]{%
    \errmessage{(Inkscape) Transparency is used (non-zero) for the text in Inkscape, but the package 'transparent.sty' is not loaded}%
    \renewcommand\transparent[1]{}%
  }%
  \providecommand\rotatebox[2]{#2}%
  \newcommand*\fsize{\dimexpr\f@size pt\relax}%
  \newcommand*\lineheight[1]{\fontsize{\fsize}{#1\fsize}\selectfont}%
  \ifx\svgwidth\undefined%
    \setlength{\unitlength}{100.06739044bp}%
    \ifx\svgscale\undefined%
      \relax%
    \else%
      \setlength{\unitlength}{\unitlength * \real{\svgscale}}%
    \fi%
  \else%
    \setlength{\unitlength}{\svgwidth}%
  \fi%
  \global\let\svgwidth\undefined%
  \global\let\svgscale\undefined%
  \makeatother%
  \begin{picture}(1,1.11155544)%
    \lineheight{1}%
    \setlength\tabcolsep{0pt}%
    \put(0,0){\includegraphics[width=\unitlength]{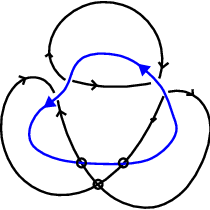}}%
    \put(0.1,-0.07){\color[rgb]{0,0,0}\makebox(0,0)[lt]{\lineheight{1.25}\smash{\begin{tabular}[t]{l}$A^{-2}\cdot d$\end{tabular}}}}%
  \end{picture}%
\endgroup%
 
    \end{tabular}
    \caption{The $\Zh$-states and weights for the right-handed virtual trefoil.}
    \label{fig:tref_zh_states}
\end{figure}
   
\subsection{The$\Zh$-polynomial $\&$ the DKM-polynomial} Example \ref{example_v_tref_zh_poly} is generic. Here we show that the $\Zh$-polynomial is equal to the DKM-polynomial after the obvious substitution $Z_i \to K_i$ for all $i$. This is an immediate consequence of the following lemma.

\begin{lemma}
Let $D$ be a virtual link diagram. Let $C$ simultaneously denote a closed curve in the $\Zh$-state $S$ for the $\Zh$-bracket and the corresponding arrow state in the DKM-bracket. Give $C$ an arbitrary orientation $\vec{C}$. Then the contribution of $C$ to $[S]$ in $\langle\langle D \rangle\rangle$ is $K_{|\vlk(\omega^{\text{op}},\vec{C})|/2}$.  
\label{main_lemma}
\end{lemma}

\begin{theorem} \label{thm_zh_eq_DKM} The $\Zh$- and DKM-polynomials are equivalent virtual link invariants.
\end{theorem}

\begin{proof} First we place the $\omega^{\text{op}}$-component in a useful position relative to $D$. As discussed in Section \ref{sec_zh_defn}, it may be assumed that the $\omega^{\text{op}}$-component has been drawn so that the only crossings involving $\omega^{\text{op}}$ are classical over-crossings with $D$ and virtual crossings with $D$. Thus, there is a well-defined right and left side of $\omega^{\text{op}}$. Furthermore, it may be supposed that $\omega^{\text{op}}$ has been drawn so that all classical crossing of $D$ are to the right of $\omega$ (see Figure \ref{zh_construction}). 
\newline
\newline
Now, let $S$ be a $\Zh$-state of the diagram $D$. Observe that the poles and over-crossings of $D$ by $\omega^{\text{op}}$ can be placed in one-to-one correspondence. For an oriented smoothing at a crossing of $D$, the $\omega^{\text{op}}$-component can be pulled off of $D$ using an $\Omega 2$ move. See Figure \ref{fig:zh_poles}, left. Hence, all such over-crossing arcs can be removed in pairs. This leaves only over-crossing arcs at the unoriented smoothings of $S$. At an unoriented smoothing, there is exactly one over-crossing arc of $\omega^{\text{op}}$ for each pole. See Figure \ref{fig:zh_poles}, right. This gives the desired bijection between poles and $\omega^{\text{op}}$ over-crossings. 
\newline

\begin{figure}
    \centering
    \includegraphics[scale=1.2]{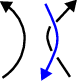}\quad \quad \quad \includegraphics[scale=1.2]{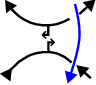}
    \caption{The correspondence between $\omega^{\text{op}}$ over-crossings and poles.}
    \label{fig:zh_poles}
\end{figure}

 Now consider two \textit{consecutive} poles as one walks along a fixed component $C$ of the $\Zh$-state $S$. Due to the alternating orientation of the arcs between poles in the state of $\langle\langle D \rangle\rangle$, there are only four possible configurations of consecutive canceling poles. These are shown in Figure \ref{fig:canceling_poles}, where the consecutive poles on $C$ are connected by a dotted black arc. Let $\vec{C}$ be an arbitrary orientation on $C$. From Figure \ref{fig:canceling_poles}, we see that for any orientation $\vec{C}$ of $C$, the contribution to $\vlk(\omega^{\text{op}},\vec{C})$ for the two over-crossing arcs is zero. Similarly, following the alternating orientation, there are four configurations of consecutive oriented poles that do not cancel. These are shown in Figure \ref{fig:noncanceling_poles}. In this case, the contribution to $\vlk(\omega^{\text{op}},\vec{C})$ from the over-crossing arcs is $\pm 2$. Hence, the total contribution of $C$ to $[S]$ is $K_{|\vlk(\omega^{\text{op}},\vec{C})|/2}$. This completes the proof of the Lemma \ref{main_lemma} and, consequently, the proof of Theorem \ref{thm_zh_eq_DKM}.\end{proof}
 
\begin{figure}[htb]
    \centering
    \includegraphics[scale=1.1]{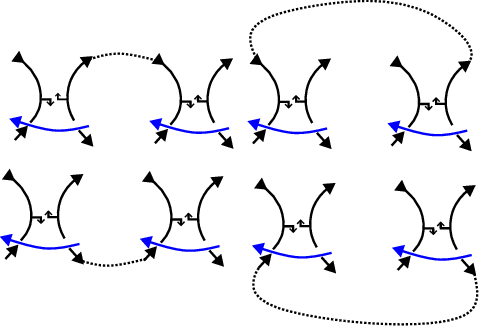} 
    \caption{The four configurations of canceling poles.}
    \label{fig:canceling_poles}
\end{figure}

\begin{figure}[htb]
    \centering
    \includegraphics[scale=1.1]{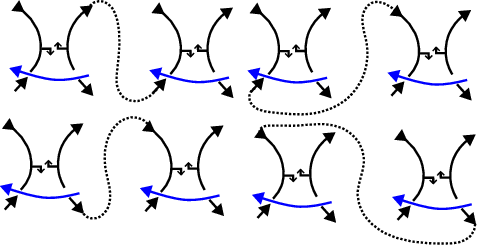} 
    \caption{The four configurations of non-canceling poles.}
     \label{fig:noncanceling_poles}
\end{figure}

\subsection{Applications} Now we apply the identification of the $\Zh$-bracket and the DKM-bracket to give simple proofs of several well-known facts about the DKM-polynomial. Recall that a summand of $\langle \langle D \rangle \rangle$ is called a \emph{surviving state}. Each surviving state has the form $A^{m} K^{j_{1}}_{i_{1}} \ldots K^{j_{n}}_{i_{n}}$. The $k$-\emph{degree} of a surviving state is defined to be the integer $ \sum_{l=1}^n i_{l} \cdot j_{l}$. The $k$-degree corresponds to the half number of oriented poles on some arrow state curve after reduction. For a virtual knot diagram $D$, let $AS(D)$ be the set of $k$-degrees of the surviving states of $D$. 

\begin{example} For the virtual trefoil $D$, Example \ref{example_v_tref_zh_poly} implies that $AS(D)=\{0,1\}$. 
\end{example}

Now consider the following result, which states that the DKM-polynomial is an extension of the Jones polynomial. Note that the extension is proper since, for example, the DKM-polynomial of the right-handed trefoil is not equal to its Jones polynomial.

\begin{theorem}[Dye-Kauffman \cite{dye2009virtual}, Miyazawa \cite{miyazawa2008}]
 For any virtual link diagram $D$, if $D$ is classical then $AS(D)=\{0\}$. Hence, if $D$ is classical, $V_{DKM}(D)$ is equal to the normalized Kauffman bracket polynomial of $D$ (i.e. the Jones polynomial).
\end{theorem}

Since every classical knot is almost classical, this result is an immediate corollary of the following stronger theorem which is easily proved using the $\Zh$-bracket. This result was also independently obtained using a different method by K. Miller (\cite{miller}, Theorem 3.2.1).

\begin{theorem}
 For any virtual link diagram $D$, if $D$ is almost classical then $AS(D)=\{0\}$ and $V_{DKM}(D)$ is equal to the normalized Kauffman bracket polynomial of $D$.
\end{theorem}
\begin{proof} If $D$ is almost classical, then $\Zh^{\text{op}}(D)$ is semi-welded equivalent to the split link $D \sqcup \bigcirc$ (see Theorem \ref{thm_splits}). Then for each $\Zh$-state curve $C$, $\vlk(\omega^{\text{op}},\vec{C})=0$. All surviving states of $\langle\langle D \rangle\rangle_{\zh}$ must therefore have $k$-degree $0$. The result follows from Theorem \ref{thm_zh_eq_DKM}.
\end{proof}

For a virtual link diagram $D$, let $v(D)$ denote the minimum number of virtual crossings among all virtual link diagrams equivalent to $D$. A useful application of the DKM-polynomial is that it gives a lower bound on $v(D)$. Here we give a simple proof of this fact using the $\Zh$-construction. 
 
\begin{theorem}[Dye-Kauffman \cite{dye2009virtual}, Miyazawa \cite{miyazawa2008}]
 For any virtual link diagram $D$, $$ \max(AS(D)) \leq v(D).$$
\end{theorem}
\begin{proof} Choose a diagram $D' \leftrightharpoons D$ that realizes the minimum number of virtual crossings. Then construct the virtual Alexander system $(D',\upsilon,\Upsilon)$ for $D'$ (see Definition \ref{defn_virtual_alex_system} and Figure \ref{fig_virt_alex_system}). By Theorem \ref{thm_main}, this is equivalent to the Alexander system for $\Zh^{\text{op}}(D)$. Hence, the virtual Alexander system can be used instead of $\Zh^{\text{op}}(D)$ to compute the DKM-polynomial. The total number of classical crossings between $D'$ and $\upsilon$ is twice the number of virtual crossings of $D$. Hence, the $k$-degree of any surviving state can be at most $v(D)$. 
\end{proof}

\section{Extending quandles with $\Zh$} \label{sec_ext_quandle}

\subsection{The Extended Fundamental Quandle} Joyce \cite{joyce} proved that the fundamental quandle of a classical knot is complete a knot invariant, up to mirror images. Quandles have played a central role in virtual knot theory ever since its inception (see Kauffman \cite{KaV}). Virtual knot invariants can be extracted from the fundamental quandle by counting the number of colorings of $D$ by a finite quandle. These can be computationally time consuming even when the size of the quandle is small. For the related problem of the coloring invariants derived from the fundamental group, we refer the reader to \cite{intractable}. Thus, it is useful to strengthen the coloring invariants obtained from a fixed finite quandle. The aim of this section is to show that the counting and 2-cocycle invariants for a finite quandle can both be strengthened by applying the $\Zh$-construction.
\newline
\newline
Recall that a \emph{quandle} is a set $X$ with a binary operation $*:X \times X\to X$ such that:
\begin{enumerate}
\item[(I)] For all $x\in X$, $x*x=x$.
\item[(II)] For all $x,y \in X$, there is a unique $z \in X$ such that $z*y=x$.
\item[(III)] For all $x,y,z \in X$, $(x*y)*z=(x*z)*(y*z)$.
\end{enumerate}
In Axiom (II), we will write $z=x\,\bar{*}\,y$. For any virtual link diagram $D$, the \emph{fundamental quandle} $Q(D)$ of $D$ is defined as follows. The generators of $Q(D)$ are in one-to-one correspondence with the arcs of $D$. The relations of $Q(D)$ are in one-to-one correspondence with the classical crossings of $D$. The classical crossing relations are shown in Figure \ref{fig_fun_quandle}. 

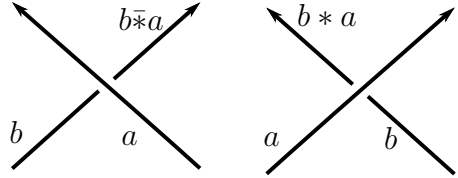
\begin{figure}[htb]
\begin{tabular}{c c}
 & \\
\begin{tabular}{c}
\def\svgwidth{1in}
\begingroup%
  \makeatletter%
  \providecommand\color[2][]{%
    \errmessage{(Inkscape) Color is used for the text in Inkscape, but the package 'color.sty' is not loaded}%
    \renewcommand\color[2][]{}%
  }%
  \providecommand\transparent[1]{%
    \errmessage{(Inkscape) Transparency is used (non-zero) for the text in Inkscape, but the package 'transparent.sty' is not loaded}%
    \renewcommand\transparent[1]{}%
  }%
  \providecommand\rotatebox[2]{#2}%
  \newcommand*\fsize{\dimexpr\f@size pt\relax}%
  \newcommand*\lineheight[1]{\fontsize{\fsize}{#1\fsize}\selectfont}%
  \ifx\svgwidth\undefined%
    \setlength{\unitlength}{157.16130855bp}%
    \ifx\svgscale\undefined%
      \relax%
    \else%
      \setlength{\unitlength}{\unitlength * \real{\svgscale}}%
    \fi%
  \else%
    \setlength{\unitlength}{\svgwidth}%
  \fi%
  \global\let\svgwidth\undefined%
  \global\let\svgscale\undefined%
  \makeatother%
  \begin{picture}(1,0.81288342)%
    \lineheight{1}%
    \setlength\tabcolsep{0pt}%
    \put(0,0){\includegraphics[width=\unitlength]{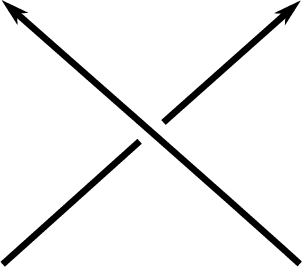}}%
    \put(-0.00528168,0.14517132){\color[rgb]{0,0,0}\makebox(0,0)[lt]{\lineheight{40.54999924}\smash{\begin{tabular}[t]{l}$b$\end{tabular}}}}%
    \put(0.58240181,0.13191535){\color[rgb]{0,0,0}\makebox(0,0)[lt]{\lineheight{40.54999924}\smash{\begin{tabular}[t]{l}$a$\end{tabular}}}}%
    \put(0.56472711,0.73727354){\color[rgb]{0,0,0}\makebox(0,0)[lt]{\lineheight{40.54999924}\smash{\begin{tabular}[t]{l}$b \bar{*} a$\end{tabular}}}}%
  \end{picture}%
\endgroup%
 \end{tabular} & \begin{tabular}{c}\def\svgwidth{1in}
\begingroup%
  \makeatletter%
  \providecommand\color[2][]{%
    \errmessage{(Inkscape) Color is used for the text in Inkscape, but the package 'color.sty' is not loaded}%
    \renewcommand\color[2][]{}%
  }%
  \providecommand\transparent[1]{%
    \errmessage{(Inkscape) Transparency is used (non-zero) for the text in Inkscape, but the package 'transparent.sty' is not loaded}%
    \renewcommand\transparent[1]{}%
  }%
  \providecommand\rotatebox[2]{#2}%
  \newcommand*\fsize{\dimexpr\f@size pt\relax}%
  \newcommand*\lineheight[1]{\fontsize{\fsize}{#1\fsize}\selectfont}%
  \ifx\svgwidth\undefined%
    \setlength{\unitlength}{146.8194203bp}%
    \ifx\svgscale\undefined%
      \relax%
    \else%
      \setlength{\unitlength}{\unitlength * \real{\svgscale}}%
    \fi%
  \else%
    \setlength{\unitlength}{\svgwidth}%
  \fi%
  \global\let\svgwidth\undefined%
  \global\let\svgscale\undefined%
  \makeatother%
  \begin{picture}(1,0.87014254)%
    \lineheight{1}%
    \setlength\tabcolsep{0pt}%
    \put(0,0){\includegraphics[width=\unitlength]{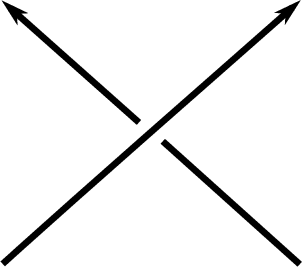}}%
    \put(-0.00565372,0.15539712){\color[rgb]{0,0,0}\makebox(0,0)[lt]{\lineheight{40.54999924}\smash{\begin{tabular}[t]{l}$a$\end{tabular}}}}%
    \put(0.6234259,0.1412074){\color[rgb]{0,0,0}\makebox(0,0)[lt]{\lineheight{40.54999924}\smash{\begin{tabular}[t]{l}$b$\end{tabular}}}}%
    \put(0.16935337,0.7844768){\color[rgb]{0,0,0}\makebox(0,0)[lt]{\lineheight{40.54999924}\smash{\begin{tabular}[t]{l}$b * a$\end{tabular}}}}%
  \end{picture}%
\endgroup%
 \end{tabular} 
\end{tabular}
\caption{Crossing relations for the fundamental quandle.} \label{fig_fun_quandle}
\end{figure}

\begin{definition} The \emph{extended fundamental quandle} of a virtual knot diagram $D$ is the fundamental quandle of $\Zh(D)$: $$\widetilde{Q}(D):=Q(\Zh(D)).$$
\end{definition}

\begin{theorem} \label{thm_ext_fun_quandle} If $D_0 \leftrightharpoons D_1$, then $\widetilde{Q}(D_0) \cong \widetilde{Q}(D_1)$ as quandles.
\end{theorem}
\begin{proof} Since $\Zh(D_0)$ and $\Zh(D_1)$ are semi-welded equivalent and the fundamental quandle of a knot is invariant under all extended Reidemeister and detour moves, it is sufficient to show that the fundamental quandle is invariant under the semi-welded move. That is, it must be checked that $Q$ is invariant under the $\omega\text{OCC}$ move (see Figure \ref{sw_equiv}). There are essentially two cases, as shown in Figure \ref{fig_extended_quandle_ok}. Both sides of the semi-welded move on the left in Figure \ref{fig_extended_quandle_ok} yield the equations $a * v=c$, $b * v=d$. On the right in Figure \ref{fig_extended_quandle_ok}, both sides of the semi-welded move yield the equations $b*v=d$, $c\, \bar{*}\, v=a$. Thus, $Q(\Zh(D_0))\cong Q(\Zh(D_1))$. 
\end{proof}

\begin{figure}[htb]
\tiny
\def\svgwidth{5.2in}
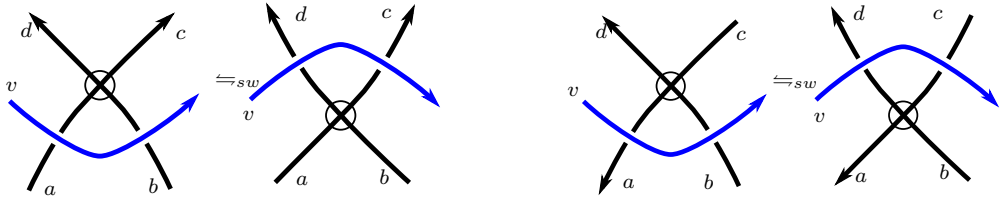 
\normalsize
\caption{Extended fundamental quandles and semi-welded equivalence.}
\label{fig_extended_quandle_ok}
\end{figure}

For the remainder of this subsection, we discuss the extent to which $\widetilde{Q}(D)$ can be interpreted as an extension of $Q(D)$. Recall that the fundamental group $G(D)$ of a virtual knot diagram $D$ is the group generated by the arcs of $D$ and having one Wirtinger relation for every classical crossing of $D$. The fundamental group of $\Zh(D)$, denoted $\widetilde{G}(D)$, is an extension of $G(D)$. The surjective homomorphism $\widetilde{G}(D) \to G(D)$ sends the generator $v$ for the the $\omega$ component to the identity in $G(D)$ (see \cite{c_ext}, Section 2.4). Note that there is no similar quandle map $\widetilde{Q}(D) \to Q(D)$ for the simple reason that quandles do not in general posses an element acting as a multiplicative identity. Thus, in order to define a surjective homomorphism with domain $\widetilde{Q}(D)$, we must add an image element for $v$ in $Q(D)$. This can be done as follows. 
\newline
\newline
Let $(X,*)$ be a quandle and $v$ a symbol not in $X$. Define a new quandle $X_v$ with underlying set $X \sqcup \{v\} $ and operation $*_v:X_v \times X_v \to X_v$. The operation is given by $*_v|X \times X=*$, $x*_v v=x$ for all $x \in X$, $v*_v x=v$ for all $x\in X$, and $v *_v v=v$. 

\begin{lemma} If $(X,*)$ is a quandle, so also is $(X_v,*_v)$.
\end{lemma}
\begin{proof} Axioms (I),(II),(III) hold for $(X_v,*_v)$ whenever $x,y,z \in X$. All that remains is to check axioms (II) and (III) when at least one of the elements in each formula is $v$. Let $x \in X$. For axiom (II), note that the unique element $z$ such that $z *_v x=v$ is $v$ and the unique element $z$ such that $z*_v v=x$ is $x$. For axiom (III), let $x,y \in X_v$. Then we have: 
\begin{eqnarray*}
(x*_v y)*_v v &=& x*_v y = (x *_v v)*_v (y *_v v)\\
(x*_v v)*_v y &=& x*_v y= (x *_v y)*_v (v *_v y)\\
(v*_v x)*_v y  &=& v = (v*_v y)*_v (x *_v y)
\end{eqnarray*}
Hence, self-distributivity is satisfied when one or more of the involved elements is $v$.\end{proof}

Now we apply this construction to the fundamental quandle of a virtual knot diagram. If $X=Q(D)$ is the fundamental quandle of a virtual knot diagram $D$ and $v$ is a symbol not in $X$, define $Q_v(D)=X_v$. Set $v$ to be the generator of $\widetilde{Q}(D)$ corresponding to the $\omega$ component.

\begin{theorem} \label{lemma_homomorphism} The extended quandle $\widetilde{Q}(D)$ is an extension of the fundamental quandle $Q(D)$ in the following sense.
\begin{enumerate}
\item There is an injective homomorphism $\iota_v^D:Q(D) \hookrightarrow Q_v(D)$ with $v \not \in \iota_v^D(Q(D))$.
\item There is a surjective homomorphism $\varphi_v^D:\widetilde{Q}(D) \twoheadrightarrow Q_v(D)$ such that $\varphi_v^D(v)=v$.
\end{enumerate}
\end{theorem}
\begin{proof} For $(1)$, note that $Q_v(D)=Q(D) \sqcup \{v\}$ as sets. Define $\iota_v^D:Q(D) \to Q_v(D)$ by $\iota_v^D(x)=x$. This is a homomorphism by definition of the operation $*_v$. If $x_0,x_1 \in Q(D)$ and $x_0 \ne x_1$, then the image of these elements lies in the subset $Q(D)$ of $Q_v(D)$. Hence $\iota_v^D(x_0)=x_0 \ne x_1=\iota_v^D(x_1)$. This completes the proof of the first claim. 
\newline
\newline
For $(2)$, Figure \ref{fig_ext_onto} shows two pictures of the form $A \to B$. The $A$ part depicts the generators near a classical crossing of $D$ for the group $\widetilde{Q}(D)$. Note that for the $A$ in the left-hand pair, the generator $d$ can be removed from the presentation. Indeed, we have the relations $c*b=d$ and $d*v=e$. Similarly, for the $A$ in the right-hand pair, the generator $b$ can be removed from the presentation for $\widetilde{Q}(D)$. After the removal of all such generators from $\widetilde{Q}(D)$, we find that that there is a presentation for $\widetilde{Q}(D)$ generated by $v$ and the short arcs of $D$. Now, every short arc $s$ of $D$ is contained in a unique arc $u_s$ of $D$. Define $\varphi_v^D$ by $\varphi_v^D(v)=v$ and for every short arc $s$ of $D$, set $\varphi_v^D(s)=u_s$. This has the effect of relabeling the arcs of $\Zh(D)$ at a crossing as depicted in Fig \ref{fig_ext_onto}. The relabeled picture is the ``target'' $B$ for the arrow $A \to B$. 
\newline
\newline
Clearly, the function $\varphi_v^D:\widetilde{Q}(D) \to Q_v(D)$ is surjective. To show it is a homomorphism of quandles, we check the crossing relations. On the left-hand side of Figure \ref{fig_ext_onto}, we have $\varphi_v^D(d*v)=\varphi_v^D(e)=e=e *_v v=\varphi_v^D(d)*_v \varphi_v^D(v)$, $\varphi_v^D(a\bar{*}v)=\varphi_v^D(b)=a=a \bar{*}_v v=\varphi_v^D(a)\bar{*}_v \varphi_v^D(v)$, and $\varphi_v^D(c*b)=\varphi_v^D(d)=e=c *_v a=\varphi_v^D(c)*_v \varphi_v^D(b)$. On the right-hand side of Figure \ref{fig_ext_onto}, we have $\varphi_v^D(d*v)=\varphi_v^D(e)=d=d *_v v=\varphi_v^D(d) *_v \varphi_v^D(v)$, $\varphi_v^D(a\bar{*} v)=\varphi_v^D(b)=a=a \bar{*}_v v=\varphi_v^D(a)\bar{*}_v \varphi_v^D(v)$, and $\varphi_v^D(b\bar{*}d)=\varphi_v^D(c)=c=a \bar{*}_v d=\varphi_v^D(b) \bar{*}_v \varphi_v^D(d)$. This completes the proof of (2). \end{proof}
 
\begin{figure}[htb]
\tiny
\def\svgwidth{5.2in}
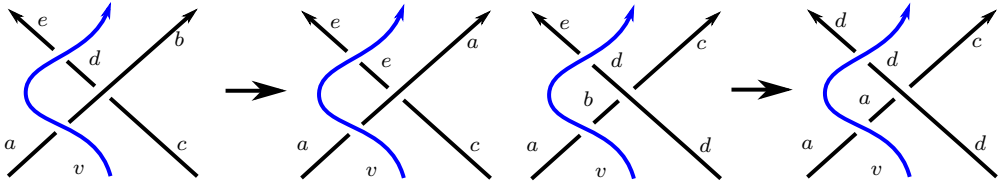 
\normalsize
\caption{The homomorphism $\varphi_v^D:\widetilde{Q}(K) \to Q_v(K)$ from Theorem \ref{lemma_homomorphism}.}
\label{fig_ext_onto}
\end{figure}

\subsection{Coloring and cocycle invariants} 
 If $X$ is any quandle and $D$ is a virtual knot diagram, a \emph{coloring by} $X$ is a quandle homomorphism from the fundamental quandle $Q(D)$ to $X$. A coloring of $D$ by a quandle labels the arcs of $D$ so that the relations at each classical crossing of $D$ are valid when interpreted in $X$. Let $X$ be a finite quandle and let $\text{Col}_X(D)=\text{Hom}(Q(D),X)$ be the set of colorings of $D$ by $X$. Then the cardinality of the finite set $|\text{Col}_X(D)|$ is a virtual knot invariant. Define the set of extended colorings of $D$ by $X$ to be $\widetilde{\text{Col}}_X(D)=\text{Col}_X(\Zh(D))$.

\begin{theorem} \label{thm_extend_cols} If $D_0,D_1$ are equivalent virtual knot diagrams and $X$ is a finite quandle, 
$$|\widetilde{\text{Col}}_X(D_0)|=|\widetilde{\text{Col}}_X(D_1)|.$$ 
\end{theorem} 
\begin{proof} It is well-known that there is a one-to-one correspondence between the colorings by $X$ of any two virtual knot diagrams related by Reidemeister or detour moves. It follows from the coloring in Figure \ref{fig_extended_quandle_ok} that there is a one-to-one correspondence between the colorings by $X$ for any two diagrams related by an $\omega\text{OCC}$ move. Thus, $|\text{Col}_X(\Zh(D_0))=|\text{Col}_X(\Zh(D_1))|$.
\end{proof}

\begin{corollary} \label{cor_extend_colors_unknot} Let $X$ be a finite quandle and $D$ a virtual link diagram. If $D$ is almost classical, then $|\widetilde{\text{Col}}_X(D)|=|\text{Col}_X(D)|\cdot|X|$. In particular, if $D$ is unknotted, then $|\widetilde{\text{Col}}_X(D)|=|X|^2$.
\end{corollary}
\begin{proof} By Theorem \ref{thm_zh_splits}, $\Zh(D)$ is semi-welded equivalent to the split diagram $D \sqcup \bigcirc$. Hence, by Theorem \ref{thm_extend_cols}, we have that $|\widetilde{\text{Col}}_X(D)|=|\text{Col}_X(D)|\cdot |X|$. 
\end{proof}

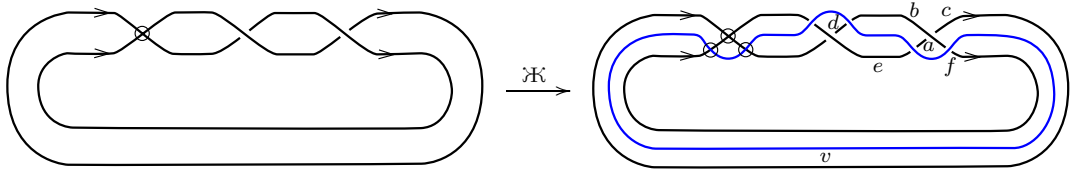
\begin{figure}[htb]
\tiny
\[
\xymatrix{\begin{array}{c}\def\svgwidth{2.5in}
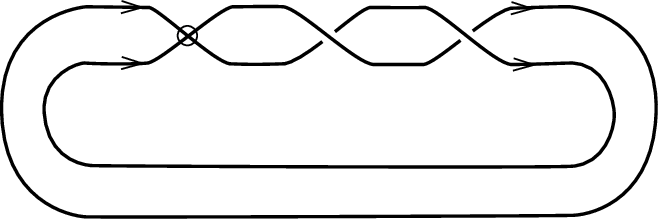 \end{array} \ar[r]^-{\Zh} & \begin{array}{c} \def\svgwidth{2.5in}
\begingroup%
  \makeatletter%
  \providecommand\color[2][]{%
    \errmessage{(Inkscape) Color is used for the text in Inkscape, but the package 'color.sty' is not loaded}%
    \renewcommand\color[2][]{}%
  }%
  \providecommand\transparent[1]{%
    \errmessage{(Inkscape) Transparency is used (non-zero) for the text in Inkscape, but the package 'transparent.sty' is not loaded}%
    \renewcommand\transparent[1]{}%
  }%
  \providecommand\rotatebox[2]{#2}%
  \newcommand*\fsize{\dimexpr\f@size pt\relax}%
  \newcommand*\lineheight[1]{\fontsize{\fsize}{#1\fsize}\selectfont}%
  \ifx\svgwidth\undefined%
    \setlength{\unitlength}{315.70829891bp}%
    \ifx\svgscale\undefined%
      \relax%
    \else%
      \setlength{\unitlength}{\unitlength * \real{\svgscale}}%
    \fi%
  \else%
    \setlength{\unitlength}{\svgwidth}%
  \fi%
  \global\let\svgwidth\undefined%
  \global\let\svgscale\undefined%
  \makeatother%
  \begin{picture}(1,0.34567056)%
    \lineheight{1}%
    \setlength\tabcolsep{0pt}%
    \put(0,0){\includegraphics[width=\unitlength]{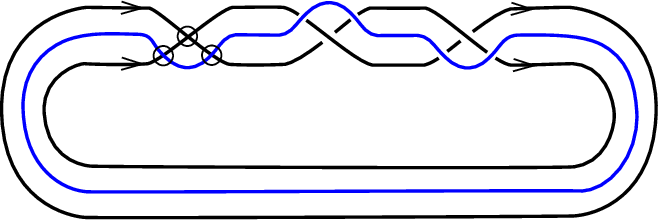}}%
    \put(0.47628235,0.01564207){\color[rgb]{0,0,0}\makebox(0,0)[lt]{\lineheight{40.54999924}\smash{\begin{tabular}[t]{l}$v$\end{tabular}}}}%
    \put(0.73573077,0.20376804){\color[rgb]{0,0,0}\makebox(0,0)[lt]{\lineheight{40.54999924}\smash{\begin{tabular}[t]{l}$f$\end{tabular}}}}%
    \put(0.73056243,0.32160518){\color[rgb]{0,0,0}\makebox(0,0)[lt]{\lineheight{40.54999924}\smash{\begin{tabular}[t]{l}$c$\end{tabular}}}}%
    \put(0.66440833,0.31953786){\color[rgb]{0,0,0}\makebox(0,0)[lt]{\lineheight{40.54999924}\smash{\begin{tabular}[t]{l}$b$\end{tabular}}}}%
    \put(0.69231705,0.24821537){\color[rgb]{0,0,0}\makebox(0,0)[lt]{\lineheight{40.54999924}\smash{\begin{tabular}[t]{l}$a$\end{tabular}}}}%
    \put(0.49282084,0.29576371){\color[rgb]{0,0,0}\makebox(0,0)[lt]{\lineheight{40.54999924}\smash{\begin{tabular}[t]{l}$d$\end{tabular}}}}%
    \put(0.58688387,0.20583537){\color[rgb]{0,0,0}\makebox(0,0)[lt]{\lineheight{40.54999924}\smash{\begin{tabular}[t]{l}$e$\end{tabular}}}}%
  \end{picture}%
\endgroup%
\end{array}} 
\]
\normalsize
\caption{A labeling the arcs of the extended fundamental quandle of the positive virtual trefoil.}
\label{fig_virt_terf}
\end{figure}

\begin{example} This example will show that the extended colorings by a finite quandle can distinguish virtual knot types even in the case that the coloring invariants for all finite quandles cannot. Let $D$ be the positive virtual trefoil knot diagram in Figure \ref{fig_virt_terf}, left. Then it is easy to see that $Q(D)$ is the trivial one element quandle. For any finite quandle $X$, it follows that $|\text{Col}_X(D)|=|X|$ and hence $D$ cannot be distinguished from the unknot using this invariant. 
\newline
\newline
A presentation for the extended fundamental quandle can be calculated from Figure \ref{fig_virt_terf}:
\[
\widetilde{Q}(D)=\langle a,b,c,d,e,f,v| c*e=d,a*b=c, e*v=a,f*v=b,d*v=b,e*v=f \rangle. 
\]
Here $v$ is the generator corresponding to $\omega$. Let $(X,*)$ be the Alexander quandle on $\mathbb{Z}_7$ with $t=3$. In other words, $x*y=3x+5y$ for all $x,y \in \mathbb{Z}_7$. The relation matrix is then given by:
\[
\begin{bmatrix}
0 & 0 & 3 & 6 & 5 & 0 & 0 \\
3 & 5 & 6 & 0 & 0 & 0 & 0 \\
6 & 0 & 0 & 0 & 3 & 0 & 5 \\
0 & 6 & 0 & 0 & 0 & 3 & 5 \\
0 & 6 & 0 & 3 & 0 & 0 & 5 \\
0 & 0 & 0 & 0 & 3 & 6 & 5
\end{bmatrix}.
\]
Over $\mathbb{Z}_7$, this matrix has nullity $1$, which means that the system of linear equations over $\mathbb{Z}_7$ has one free variable. This implies that $\widetilde{\text{Col}}_X(K)=|X|^1=7$. Then Corollary \ref{cor_extend_colors_unknot} implies that the positive virtual trefoil is not equivalent to the unknot.
\end{example}

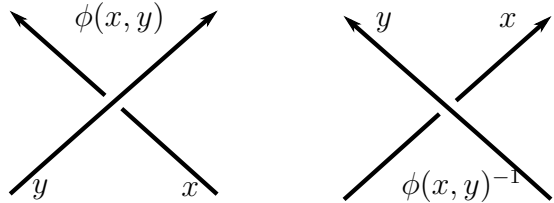
\begin{figure}[htb]
\def\svgwidth{2.85in}
\begingroup%
  \makeatletter%
  \providecommand\color[2][]{%
    \errmessage{(Inkscape) Color is used for the text in Inkscape, but the package 'color.sty' is not loaded}%
    \renewcommand\color[2][]{}%
  }%
  \providecommand\transparent[1]{%
    \errmessage{(Inkscape) Transparency is used (non-zero) for the text in Inkscape, but the package 'transparent.sty' is not loaded}%
    \renewcommand\transparent[1]{}%
  }%
  \providecommand\rotatebox[2]{#2}%
  \newcommand*\fsize{\dimexpr\f@size pt\relax}%
  \newcommand*\lineheight[1]{\fontsize{\fsize}{#1\fsize}\selectfont}%
  \ifx\svgwidth\undefined%
    \setlength{\unitlength}{374.41539669bp}%
    \ifx\svgscale\undefined%
      \relax%
    \else%
      \setlength{\unitlength}{\unitlength * \real{\svgscale}}%
    \fi%
  \else%
    \setlength{\unitlength}{\svgwidth}%
  \fi%
  \global\let\svgwidth\undefined%
  \global\let\svgscale\undefined%
  \makeatother%
  \begin{picture}(1,0.35090132)%
    \lineheight{1}%
    \setlength\tabcolsep{0pt}%
    \put(0,0){\includegraphics[width=\unitlength]{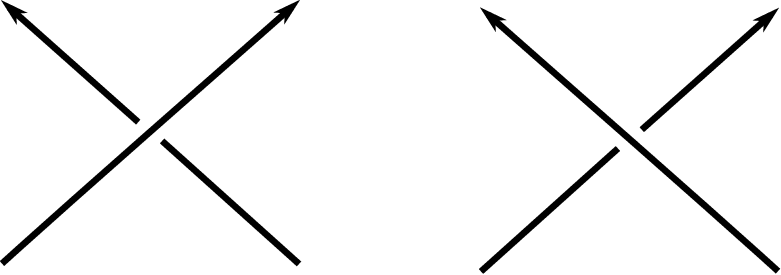}}%
    \put(0.04325068,0.01570404){\color[rgb]{0,0,0}\makebox(0,0)[lt]{\lineheight{40.54999924}\smash{\begin{tabular}[t]{l}$y$\end{tabular}}}}%
    \put(0.31739371,0.01175526){\color[rgb]{0,0,0}\makebox(0,0)[lt]{\lineheight{40.54999924}\smash{\begin{tabular}[t]{l}$x$\end{tabular}}}}%
    \put(0.90062455,0.32083862){\color[rgb]{0,0,0}\makebox(0,0)[lt]{\lineheight{40.54999924}\smash{\begin{tabular}[t]{l}$x$\end{tabular}}}}%
    \put(0.67446562,0.32245402){\color[rgb]{0,0,0}\makebox(0,0)[lt]{\lineheight{40.54999924}\smash{\begin{tabular}[t]{l}$y$\end{tabular}}}}%
    \put(0.12037594,0.32406944){\color[rgb]{0,0,0}\makebox(0,0)[lt]{\lineheight{40.54999924}\smash{\begin{tabular}[t]{l}$\phi(x,y)$\end{tabular}}}}%
    \put(0.72131262,0.01390851){\color[rgb]{0,0,0}\makebox(0,0)[lt]{\lineheight{40.54999924}\smash{\begin{tabular}[t]{l}$\phi(x,y)^{-1}$\end{tabular}}}}%
  \end{picture}%
\endgroup%
 
\normalsize
\caption{The Boltzmann weights for each classical crossing type.}
\label{fig_boltzmann}
\end{figure}

Next we will show that the 2-cocycle counting invariants of $\Zh(D)$ are also invariants of $D$. For $X$ a finite quandle and $A$ an abelian group, recall that a $2$-\emph{cocycle} is a map $\phi:X \times X \to A$ satisfying:
\[
\phi(x,z)\phi(x*z,y*z)=\phi(x*y,z)\phi(x,y), \text{ and } \phi(x,x)=1 \quad \forall x,y,z \in X,
\] 
where the operation in $A$ is written as multiplication. If $\phi$ is a $2$-cocycle of $X$, $D$ is a virtual link diagram, and $\chi$ is a coloring of $D$ by $X$, assign a \emph{Boltzmann weight} $B(\tau,\chi)$ to each classical crossing $\tau$ of $D$ according to Figure \ref{fig_boltzmann}. The cocycle counting invariant of $D$ is defined to be:
\[
\Phi_X^{\phi}(D)=\sum_{\chi \in \text{Col}_X(D)} \prod_{\tau} B(\tau,\chi)\in \mathbb{Z}[A]
\]
This was proved to be an invariant of classical links by Carter et al. \cite{CJKLS}. The same argument works for virtual links.

\begin{definition} For $X$ a finite quandle, $\phi$ a $2$-cocycle of $X$, and $D$ a virtual link diagram, the \emph{extended cocycle counting invariant} is defined to by $\widetilde{\Phi}^{\phi}_X(D)=\Phi_X^{\phi}(\Zh(D))$.
\end{definition}

\begin{theorem} \label{thm_extend_cocycle} If $D_0,D_1$ are equivalent virtual knot diagrams, then $\widetilde{\Phi}^{\phi}_X(D_0)=\widetilde{\Phi}^{\phi}_X(D_1)$.
\end{theorem}
\begin{proof} Since $D_0 \leftrightharpoons D_1$, $\Zh(D_0) \leftrightharpoons_{sw} \Zh(D_1)$. Since $\Phi_X^{\phi}$ is invariant under Reidemeister moves and detour moves, it needs only be shown that $\widetilde{\Phi}_X^{\phi}$ is invariant under the $\omega$OCC move. Suppose then that $\Zh(D_0)$ and $\Zh(D_1)$ are related by an $\omega$OCC move. By the proof of Theorem \ref{thm_ext_fun_quandle}, it follows that there is a one-to-one correspondence between the extended colorings of $\Zh(D_0)$ and $\Zh(D_1)$. From Figure \ref{fig_extended_quandle_ok}, left, the contribution of Boltzmann weights from the two classical crossings is $\phi(a,v)\phi(b,v)$ on each side of the move. Similarly, Figure \ref{fig_extended_quandle_ok}, right, gives a contribution of $\phi(a,v)^{-1} \phi(b,v)$ on each side of the move. Thus, $\widetilde{\Phi}^{\phi}_X(D_0)=\widetilde{\Phi}^{\phi}_X(D_1)$ and the theorem is proved.
\end{proof}

\begin{corollary} \label{cor_extend_cocycle_unknot} Let $X$ be a finite quandle, $D$ a virtual link diagram, and $\phi:X \times X \to A$ a $2$-cocycle. If $D$ is almost classical, then $\widetilde{\Phi}_X^{\phi}(D)=|X|\cdot \Phi_X^{\phi}(D)$. In particular, if $D$ is equivalent to the unknot, then $\widetilde{\Phi}_X^{\phi}(D)=|X|^2 \cdot 1 \in \mathbb{Z}[A]$.
\end{corollary}
\begin{proof} By Theorem \ref{thm_zh_splits}, $\Zh(D)$ is semi-welded equivalent to the split diagram $D \sqcup \bigcirc$. Hence, by Theorem \ref{thm_extend_cocycle}, we have that $\widetilde{\Phi}_X^{\phi}(D)=|X|\cdot \Phi_X^{\phi}(D)$. 
\end{proof}

\begin{example} Again consider the virtual trefoil in Figure \ref{fig_virt_terf}. Since its fundamental quandle is the trivial one element quandle, any finite quandle $X$ and $2$-cocycle $\phi:X \times X \to A$ of $X$ will have $\Phi^{\phi}_X(D)=|X|\cdot 1 \in \mathbb{Z}[A]$. Hence, $D$ cannot be distinguished from the unknot using any coloring invariant or cocycle counting invariant. We will show that $D$ can be distinguished from the unknot using an extended $2$-cocycle invariant $\widetilde{\Phi}_X^{\phi}(D)$, even when the extended coloring invariant $|\widetilde{\text{Col}}_X(D)|$ does not. Let $(X,*)$ be the dihedral quandle on $\mathbb{Z}_4$. Then $X=\mathbb{Z}_4$ and $x*y=-x+2y \pmod{4}$ for all $x,y \in \mathbb{Z}_4$. The multiplication table is:
\[
\begin{bmatrix}
 0 & 2 & 0 & 2 \\
 3 & 1 & 3 & 1 \\
 2 & 0 & 2 & 0 \\
 1 & 3 & 1 & 3 \\
\end{bmatrix}
\]
In this case, there are 16 colorings of $\Zh(K)$ by $X$, so that $|\widetilde{\text{Col}}_X(K)|=|X|^2$ and the extended coloring invariant does not distinguish $K$ from the unknot. The colorings are:

\[
\begin{array}{|c||c|c|c|c|c|c|c|c|c|c|c|c|c|c|c|c|} \hline 
     & 1 & 2 & 3 & 4 & 5 & 6 & 7 & 8 & 9 & 10 & 11 & 12 & 13 & 14 & 15 & 16 \\ \hline \hline
a     & 0 & 0& 0& 0 & 1& 1& 1& 1 &2 &2 &2 &2 &3 &3 &3 &3 \\ \hline
b     &0 &0&2&2&1&1&3&3&0&0&2&2&1&1&3&3\\ \hline
c     & 0&0&0&0&1&1&1&1&2&2&2&2&3&3&3&3\\ \hline
d     & 0&0&0&0&1&1&1&1&2&2&2&2&3&3&3&3\\ \hline
e     &0 &0&2&2&1&1&3&3&0&0&2&2&1&1&3&3\\ \hline
f     & 0&0&0&0&1&1&1&1&2&2&2&2&3&3&3&3\\ \hline
v     & 0&2&1&3&1&3&0&2&1&3&0&2&0&2&1&3\\ \hline
\end{array}
\]
We will use the $2$-cocycle $\phi:X \times X \to A$ from Carter et al. \cite{CJKLS} (see also Elhamdadi-Nelson \cite{EN}, Example 130): 
\[
\phi(x,y)=\left\{\begin{array}{cc} u & \text{if } (x,y)=(0,1) \text{ or } (0,3) \\ 1 & \text{else} \end{array} \right.
\]
Each extended coloring $\chi$ gives a term $\prod_{\tau} B(\tau,\chi)$:
\[
\phi(f,v)^{-1}\phi(a,b)\phi(e,v)\phi(c,e)\phi(e,v)^{-1}\phi(d,v).
\]
Calculating this product for each coloring above gives $\widetilde{\Phi}_X^{\phi}(K)=14+2u$. As claimed, extended $2$-cocycle invariants can distinguish virtual knot types even when the quandle coloring invariants, quandle $2$-cocycle invariants, and the extended quandle coloring invariants are unable to do so. \emph{Mathematica} \cite{wolfram} code for the this calculation can be found at the \emph{arXiv} webpage for this paper. 
\end{example}

\section{Further Discussion} \label{sec_further}

\subsection{Quantum invariants $\&$ categorification} As is well-known, the Jones polynomial can be expressed as the normalized quantum invariant associated to $U_q(sl_2)$. The DKM-polynomial is an extension of the Jones polynomial in the sense that the DKM-polynomial of a classical knot is equal to its Jones polynomial. In Theorem \ref{thm_zh_eq_DKM}, it was shown that the DKM-polynomial can be naturally obtained from the $\Zh$-construction. Is there a general method for extending quantum or skein theoretic invariants to virtual links using the $\Zh$-construction?
\newline
\newline
A categorification of the DKM-polynomial was given by Dye-Kauffman-Manturov \cite{dkm}. In \cite{tagami}, Tagami showed how to categorify a two-variable version of the DKM-polynomial via a generalization of the Bar-Natan geometric complex. Note that both of these categorifications require coefficients in the field $\mathbb{Z}_2$. Theorem \ref{thm_zh_eq_DKM} realizes the additional variables in the DKM-polynomial with virtual linking numbers. In particular, the variables for each state can be calculated without reducing the poles in the state curves. This simplification suggests that the $\Zh$-bracket might allow for a more straightforward categorification of the DKM-polynomial.

\subsection{Comparing quandle extensions} Extensions of quandles have been heavily studied in the literature. How is the extension $\varphi_v^D:\widetilde{Q}(D) \twoheadrightarrow Q_v(D)$ related to the general theory of extensions of quandles? We mention just two comparisons with other types of quandle extensions. Carter-Kamada-Saito \cite{CKS03} defined for each quandle $(X,*)$ and $2$-cocycle $\phi:X \times X \to A$ a quandle $E(X,A,\phi)$ with underlying set $A \times X$ and operation $\circ:(A \times X) \times (A \times X) \to A \times X$ given by $(a_1, x_1) \circ (a_2, x_2) = (a_1 + \phi(x_1, x_2), x_1 * x_2)$. Here, $(A, +)$ is an abelian group. This is called a \emph{central abelian extension}. Projection onto the second factor $\pi_1:E(X,A,\phi) \to X$ is a surjective quandle homomorphism. Is $\widetilde{Q}(D)$ a central abelian extension $E(Q_v(D),A,\phi)$ for some abelian group $A$ and $2$-cocycle $\phi:Q_v(D) \times Q_v(D) \to A$?  
\newline
\newline
Eisermann \cite{eisermann_coverings}, considered the more general situation of extending quandles by a group $G$, with $G$ not necessarily abelian. Can the extended quandle $\widetilde{Q}(D)$ be realized as a quandle extension of $Q_v(D)$ in this sense? Note that all such extensions are necessarily quandle coverings $p:Q' \to Q$ (see \cite{eisermann_coverings}, Proposition 4.15). Recall that a \emph{quandle covering} is a surjective quandle homomorphism $p:Q' \to Q$ such that $p( x') = p(y')$ implies that $x',y'$ are \emph{behaviorally equivalent}: $a'*x' = a'*y' $ for all $a',x',y' \in Q'$. For example, a central abelian extension is a quandle covering. In this case, the behaviorally equivalent elements take the form $(a_1,x),(a_2,x)$ for $a_1,a_2 \in A$ and $x \in X$. If $(a,y) \in E(X,A,\phi)$, then $(a,y) \circ (a_1,x)=(a+\phi(y,x),y*x)=(a,y) \circ (a_2,x)$. Is the surjection $\varphi_v^D:\widetilde{Q}(D) \to Q_v(D)$ even a quandle covering?

\begin{acknowledgments*}
The authors would like to thank S. Mukherjee for several helpful conversations about quandle coloring and quandle 2-cocycle invariants. We are also indebted to H. A. Dye and A. Kaestner for sharing their knowledge about the DKM-polynomial. The first named author was partially supported by funds and release time from The Ohio State University, Marion campus.
\end{acknowledgments*}


\bibliographystyle{plain}
\bibliography{mil_plus_zh_bib}

\end{document}